\documentclass[11pt]{amsart}
\usepackage{amsmath,amssymb,amsbsy,amsfonts,amsthm,latexsym,amsopn,amstext,amsxtra,epic, euscript,amscd,indentfirst,verbatim,graphicx, hyperref,xcolor, setspace,tikz,tikz-cd}
\usepackage{tabularx} 
\usepackage{relsize} 
\usepackage[margin=1.2in]{geometry}
\usepackage{enumitem} 
\usepackage{hyperref} 
\DeclareMathOperator{\Tr}{Tr}
\DeclareMathOperator{\Ric}{Ric}

\begin{document}

\newtheorem{theorem}{Theorem}
\newtheorem*{theorem*}{Theorem}
\newtheorem{conjecture}[theorem]{Conjecture}
\newtheorem*{conjecture*}{Conjecture}
\newtheorem{proposition}[theorem]{Proposition}
\newtheorem*{proposition*}{Proposition}
\newtheorem{question}{Question}
\newtheorem{lemma}[theorem]{Lemma}
\newtheorem*{lemma*}{Lemma}
\newtheorem{cor}[theorem]{Corollary}
\newtheorem*{obs*}{Observation}
\newtheorem{obs}{Observation}
\newtheorem{condition}{Condition}
\newtheorem{definition}{Definition}
\newtheorem*{definition*}{Definition}
\newtheorem{proc}[theorem]{Procedure}
\newcommand{\comments}[1]{} 
\def\Z{\mathbb Z}
\def\Za{\mathbb Z^\ast}
\def\Fq{{\mathbb F}_q}
\def\R{\mathbb R}
\def\N{\mathbb N}
\def\C{\mathbb C}
\def\k{\kappa}
\def\grad{\nabla}
\def\M{\mathcal{M}}
\def\S{\mathcal{S}}
\def\pt{\partial}

\newcommand{\todo}[1]{\textbf{\textcolor{red}{[To Do: #1]}}}

\title[Hall of Statistical Mirrors]{A Hall of Statistical Mirrors}
 \author[Iowa State University]{Gabriel Khan} 
 \author[University of Michigan]{Jun Zhang} 
 
\email{gkhan@iastate.edu}
\email{junz@umich.edu}

\date{\today}

\maketitle 

\begin{abstract}
    The primary objects of study in information geometry are \emph{statistical manifolds}, which are parametrized families of probability measures, induced with the Fisher-Rao metric and a pair of torsion-free conjugate connections. In recent work \cite{zhang2020statistical}, the authors considered parametrized probability distributions as \emph{partially-flat statistical manifolds admitting torsion} and showed that there is a complex to symplectic duality on the tangent bundles of such manifolds, based on the dualistic geometry of the underlying manifold.
    
     In this paper, we explore this correspondence further in the context of Hessian manifolds, in which case the conjugate connections are both curvature- and torsion-free, and the associated dual pair of spaces are K\"ahler manifolds. We focus on several key examples and their geometric features. In particular, we show that the moduli space of univariate normal distributions gives rise to a correspondence between a Siegel domain and the Siegel-Jacobi space, which are spaces that appear in the context of automorphic forms.
\end{abstract}



\section{Introduction}

In complex and symplectic geometry, mirror symmetry is a duality between Calabi-Yau manifolds, in which two distinct manifolds have closely related geometry. This correspondence has played an important role in string theory as well enumerative algebraic geometry (see \cite{cox1999mirror,morrison1997making} for more details).
Mirror symmetry provides a correspondence between the geometry of two different Calabi-Yau manifolds $\mathbb{M}$ and $\mathbb{W}$. Such manifolds are generally topologically and geometrically distinct, so this correspondence is not induced by a mapping between the manifolds. Instead, it is more appropriate to think that the spaces appear in pairs where it is possible to understand some aspects of the complex geometry of the primal space in terms of the symplectic geometry of the dual manifold.

At present, mirror symmetry is a somewhat mysterious phenomena; given a Calabi-Yau manifold it is not clear how to construct its mirror. However, it seems to be deeply linked with T-duality,\footnote{The notion of T-duality appears both in mathematics as well as physics, where two seemingly different physical systems turn out to be equivalent. In this paper, we adopt the mathematical definition of T-duality given in \cite{leung2000mirror,fei2019anomaly}.} which is a duality between semi-flat K\"ahler manifolds induced by inverting the length-scale of the fibers. In particular, the SYZ conjecture states that away from a singular locus, the mirror manifold can be constructed using T-duality \cite{strominger1996mirror}.

In this paper, we study a related duality in the context of the tangent bundles of statistical manifolds.\footnote{Strictly speaking, non-K\"ahler statistical mirror symmetry is defined for partially-flat statistical manifolds admitting torsion, rather than traditional statistical manifolds, which are torsion-free (see Section \ref{nonK_mirror} for details).} Due to the similarities to the semi-flat case of Calabi-Yau manifolds, we call this correspondence \textit{statistical mirror symmetry} \cite{zhang2020statistical}. 
From a topological and symplectic perspective, this correspondence is often very simple. However, from the perspective of complex geometry, rich phenomena arise.

The goal of this paper is to explore the geometric properties of statistical mirror symmetry in more detail. We will focus primarily on examples whose underlying Riemannian manifold is hyperbolic and where the mirror pair are K\"ahler manifolds, and use these examples to motivate the more general results.

\subsection{Overview of the paper}

In Section \ref{Normaldistributions}, we discuss the moduli space of normal distributions (i.e., Gaussian distributions) and use this example to introduce statistical mirror symmetry. This induces a duality between a Siegel domain with its metric of constant holomorphic sectional curvature and the Siegel-Jacobi space with a $Sp(2,\mathbb{R})$-invariant metric.
We study the geometry of these two spaces in depth, and use them to highlight several important geometric features that differ from Calabi-Yau manifolds. Furthermore, we show that both of these spaces are K\"ahler-Ricci solitons which remain coupled for all time under K\"ahler-Ricci flow. Using this example as a guide, we find two curvature properties which are preserved by conjugate flow (i.e., the dual to K\"ahler-Ricci flow). Finally, we make some speculative remarks the relationship between this duality and the automorphic forms on these two spaces. In particular, we propose that this perspective may be used to explain the so-called Saito-Kurokawa lift, which takes elliptic modular forms to Siegel modular forms of degree 2.

In Section \ref{Construction}, we provide a more general definition for statistical mirror symmetry. In this section, we describe how this construction is related to the duality of semi-flat Calabi-Yau manifolds.
 We also discuss a non-K\"ahler generalization of this concept, which is defined on the tangent bundle of an affine Riemannian manifold.

In Section \ref{Hyperbolic}, we study other examples of statistical mirror pairs whose underlying statistical manifold has constant negative sectional curvature. These examples show several important principles about statistical mirror symmetry. For instance, by explicitly constructing a second example of a mirror pair (induced from the moduli of negative trinomial distributions) where one of the metrics has constant holomorphic sectional curvature, we show that two tube domains can be locally holomorphically isometric, but have non-isometric dual metrics. These examples also show that the dual of a space of constant holomorphic sectional curvature need not even have constant scalar curvature. We will also use this geometric perspective to provide new statistical results. In particular, by studying the moduli of Inverse Gaussian distributions, we compute the isometry group of the parameter space. 

In Section \ref{Flat}, we discuss statistical mirror pairs whose underlying Hessian manifolds are Riemannian flat, which are better known as Frobenius manifolds. These spaces can be used to construct pairs of solutions to the Witten-Dijkgraaf-Verlinde-Verlinde (WDVV) equations \cite{kito1999hessian}, so have applications in theoretic physics. We use our geometric approach to establish several basic properties. For instance, by considering the scalar curvature, we show that it is possible to simplify the WDVV equations to a single equation in dimension two. Using the Ricci curvature, we show that the full system of equations simplifies to six equations in dimension three. We also provide several examples of such metrics.

Finally, in Section \ref{Existence of Hessian metrics}, we discuss the question of whether a given Riemannian metric admits a Hessian structure, from both a global and local perspective. We also discuss the fact that such structures fail to be unique; metrics which admit one Hessian structure generally admit many others as well. 



\section{The univariate normal and its corresponding mirrors}
\label{Normaldistributions}

To introduce statistical mirror symmetry, it is helpful to consider a specific example of the phenomena. In this section, we show that by considering the family of univariate normal distributions, we obtain a correspondence between the Siegel-Jacobi space $\mathbb{C}\times\mathbb{H}$ with an $Sp(2,\mathbb{R})$-invariant metric and a Siegel domain (i.e., the complex ball) with the Bergman metric.  We will use this as an archetypal case which motivates the more general definition in Section \ref{Construction}.
In the interest of clarity, we will not include long derivations of curvature tensors in this paper. Instead, we have written a Mathematica notebook which can be used to calculate all of the relevant quantities \cite{MTWnotebook}.

\subsection{The geometry of the normal family}

A univariate Gaussian distribution is a probability distribution $\rho$ of the form
\[ \rho(\zeta| \mu,\sigma) = \frac{1}{\sqrt{2 \pi \sigma}} \exp \left(-\frac{(\zeta-\mu)^2}{2 \sigma^2} \right). \]
In this expression, $\zeta$ represents the random variable, which takes values in the sample space $\S = \R$. There are two parameters, $\mu$ and $\sigma$ which correspond to the mean and variance of the distribution, respectively. We can consider the space of all univariate normal distributions over the reals as a 
 \textit{parametrized family}, which is a family of probability distributions specified by some number of parameters (in this case, $\mu$ and $\sigma$). Furthermore, we can consider the space of all univariate normal distributions as a \textit{statistical manifold}, where the parameters serve as a global coordinate chart. This gives this space the structure of a smooth manifold.\footnote{In fact, it is a domain in Euclidean space, but it is beneficial to consider it as a manifold.}

For any parametrized family of probability distributions, it is possible to define an associated Riemannian metric, which is known as the \textit{Fisher metric} (or \textit{Fisher-Rao metric}. Denoting the parameters of the family of probability density functions as $ \{ x \}_{i=1}^n$, the Fisher metric is given by the expression\footnote{Here, the reference measure $d\zeta$ is the usual Lebesgue measure. } 
\[g_{jk}(x) =
\int_{\S}
 \frac{\partial \log p(\zeta|x)}{\partial x^j}
 \frac{\partial \log p(\zeta|x)}{\partial x^k}
 p(\zeta|x) \, d\zeta. \]

This expression originates in statistics, where it is also known as the {\it Fisher information}. It can be interpreted as the infinitesimal form of the relative entropy, which provides a non-symmetric notion of the distance (also known as a \textit{divergence}) between two mutually absolutely continuous probability measures. The Fisher metric has found many uses in physics and other areas of science, but a full discussion of its properties would take us too far from our main focus, so we refer the interested reader to \cite{amari2000methods} \cite{costa2015fisher}.

For the space of normal distributions (which we denote $\M_{\tt Normal}$), it is possible to compute the Fisher metric explicitly in the $(\mu,\sigma)$-coordinates. Doing so, we find that
\[g= \frac{1}{\sigma^2} (d\mu^2+2 d\sigma^2).  \]
This computation shows that the moduli of normal distributions is a hyperbolic space.

\begin{proposition}[Amari \cite{amari1980theory}]
The statistical manifold of normal distributions $\M_{\tt Normal}$ is a half-plane with constant negative curvature.
\end{proposition}

\subsection{Normal distributions as an exponential family}

In order to construct a pair of mirror K\"ahler manifolds from the space of normal distributions, we first must reparametrize $\M_{\tt Normal}$ as an \textit{exponential family}.

\begin{definition}[Exponential family]
Given a sample space $\S$ with random variable $\zeta$, an exponential family is a parametrized family of probability distributions whose probability density/mass functions are of the form
\begin{equation} \label{Expfamily}
\rho (\zeta \,| \, x) = h(\zeta) \exp\left( \sum_{i=1}^n x^i F_i(\zeta) - \Phi(x) \right).
\end{equation}
Here $h:\S \to \mathbb{R}$ is a known function which serves to fix a base measure on $S$. The parameters are denoted by the $x = (x^1, \cdots, x^n ) $ and take values in some domain $ \Omega \subset \mathbb{R}^n$. When an exponential family is parametrized in this way, the $x^i$'s are known as the \emph{natural parameters}. The functions $F=(F_1, \cdots, F_n): \S \to \mathbb{R}^n$ are known as the \emph{sufficient statistics}. Finally, the function $\Phi: \Omega \to \mathbb{R}$ is known as the \emph{log-partition function}, which serves to renormalize the distribution so that the total mass is one.
\end{definition}


In the context of normal distributions, 
we want to find functions $x^1$ and $x^2$ which are functions of $\mu$ and $\sigma$ so that the probability density function takes the form \ref{Expfamily}. To this end, we set $F = (F_1, F_2)$ to be \begin{equation}
    F_1(\zeta)=\zeta, \,\, F_2(\zeta)=\zeta^2 
\end{equation}  
and
\begin{equation} \label{Transition map from natural to mu-sigma}
x^1 = \frac{\mu}{\sigma^2} \,\, \textrm{ and } \,\, x^2 = -\frac{1}{2\sigma^2}. \end{equation}
These functions are defined on the domain 
$ \Omega = \{(x^1,x^2) \in \R^2 ~|~ x^2< 0\}$.

Finally, the function $\Phi(x)$ is the following:
\[  \Phi= -\frac{x^1 \cdot x^1}{4 x^2} - \frac{1}{2}\log\left( - \frac{x^2}{\pi} \right). \]


Exponential families play an essential role throughout this paper, so we mention some general properties about them now.

\begin{obs*}[Important facts about exponential families, part I] 
\label{important facts about exp families 1}
For any exponential family, the natural parameters have the following properties.
\begin{enumerate}
    \item The domain $\Omega$ of the natural parameters is a convex subset of $\mathbb{R}^n$.
    \item When parametrized in terms of the natural parameters $x$, the Fisher metric of an exponential family is given by the Hessian of the log-partition function $\Phi$. That is to say, we have
\[g_{ij} = \frac{\partial^2}{\partial x^i \partial x^j} \Phi. \] 
\end{enumerate}

\end{obs*}

Specializing to $\M_{\tt Normal}$, this means that when we use the $x$-coordinates, the Fisher metric is given by
\begin{eqnarray*}
g  & = &
\frac{\partial^2}{\partial x^i \partial x^j} \left( \frac{-x^1 \cdot x^1}{4 x^2} - \frac{1}{2}\log\left( - \frac{x^2}{\pi} \right) \right) \\
& = & \frac{1}{2 x^2} \left[
\begin{array}{cc}
-1  & \frac{x^1}{x^2}   \\
\frac{x^1}{ x^2} & \frac{-x^1\cdot x^1+x^2}{ x^2 \cdot x^2} \\
\end{array}
\right].
\end{eqnarray*}

In some sense, the natural parameters are ``preferred coordinates" for an exponential family,\footnote{For geometers, the notion of preferred coordinates might be somewhat anathema. Conceptually, one can rephrase this definition in terms of a flat affine connection, in which case these coordinates are those for which the Christoffel symbols vanish.} and give the manifold $\M_{\tt Normal}$ the structure of a \textit{Hessian manifold}.\footnote{We will later give another definition for Hessian manifolds that does not use coordinates (Definition \ref{Hessian manifold connection definition}).}

\begin{definition} \label{Hessian manifold-coordinate definition}
A Riemannian manifold $(\M, g)$ is a Hessian manifold if
\begin{enumerate}
    \item $\Omega$ admits an atlas of coordinate charts, such that
    \item the transition maps between charts in this atlas are affine functions, and
    \item in each coordinate chart, there is a potential function $\Phi: \Omega \to \mathbb{R}$, such that
    \[ g = \frac{\partial^2}{\partial x^i \partial x^j} \Phi. \]
\end{enumerate}
\end{definition}

For reasons that will later become apparent, we will denote Hessian manifolds using the notation $(\M,g,D)$, where $D$ indicates the affine structure.\footnote{More precisely, $D$ is a flat affine connection which satisfies a compatibility condition with $g$ known as \emph{Codazzi coupling}.}

\subsubsection{The dual parametrization}


For an exponential family, the natural parameters provide an important coordinate system. There is also a dual set of coordinates, called the \textit{expectation parameters}, which are induced by the sufficient statistics. 

\begin{obs*}[Important facts about exponential families, part II]
\begin{enumerate}
\item[]
    \item Given an exponential family $\rho( \zeta|x)$, the expected value of the sufficient statistics 
    \[ u_i = \int_{\S} F_i(\zeta) \rho ( \zeta|x) \, d\zeta \] also form a coordinate chart for the exponential family.
    \item The domain $\Omega^\ast$ of the expectation parameters $u$ is a convex subset of $\mathbb{R}^n$.
    \item If we use these expected values as coordinates, the Fisher metric is given by the Hessian of the Legendre dual $\Phi^\ast$ of the log-partition function $\Phi$. In other words, we have
\[g\left( \frac{\partial}{\partial u_i}, \frac{\partial}{\partial u_j} \right)  = \frac{\partial^2}{\partial u^i \partial u^j} \Phi^\ast, \]
where
\[ \Phi^\ast = \sup_{x \in \Omega} \langle x,u \rangle - \Phi(x). \]

\item The coordinate systems $x$ and $u$ are bi-orthogonal, which means they satisfy the identity
\[ g \left( \frac{\partial}{\partial x^i}, \frac{\partial}{\partial u_j} \right) = \delta^i_{j}. \]
\end{enumerate}
\end{obs*}

To construct the expectation parameters for normal distributions, we start by considering the sufficient statistics $F(\zeta) = (\zeta, \zeta^2)$. To translate these quantities into coordinates $u = (u_1, u_2)$ for the statistical manifold $\M_{\tt Normal}$, we compute the expected value of $F$ for a given normal distribution $\rho (\zeta \,| \, x)$:
\[
u_i = \int \rho (\zeta \,| \, x) F_i (\zeta) d\zeta.
\]

Computing these explicitly in terms of the mean and variance, we find that \[ u_1 = \mu \qquad u_2 = \mu^2 + \sigma^2. \]

As such, each normal distribution corresponds to a unique point in $u$-space, which means that we can use the expectation parameters as coordinates for the statistical manifold. The $u$ coordinates are defined on the set $\Omega^\ast = \{(u_1,u_2) \in \R^2  ~|~u_1\cdot u_1 - u_2 <0 \}$.
The Legendre dual of the log-partition function $\Phi$ is 
\[ \Phi^\ast = -\frac{1}{2}-\frac{1}{2}\log( u_2 -u_1\cdot u_1). \]
As such, in the $u$-coordinates, the Fisher metric is given by
\[ g = \frac{1}{(u^1 \cdot u^1 - u^2)^2} \left[
\begin{array}{cc}
u^1 \cdot u^1 + u^2  & - u^1  \\
-u^1 & \frac{1}{2} \\
\end{array}
\right]. \]

\subsubsection{The symmetries of the normal manifold}

Before constructing a statistical mirror pair from the moduli space of normal distributions, let us discuss the isometry group of the $\M_{\tt Normal}$ and its statistical meaning. It is a classic fact in hyperbolic geometry that the symmetry group of the half-plane is the projective special linear group $PSL(2, \mathbb{R})$, which acts by M\"{o}bius transformations.
\begin{equation} \label{Mobius transformation}
    z \mapsto \frac{a z + b}{c z + d}, 
\end{equation}
where $a,b,c,d \in \mathbb{R}$ with $ad-bc=1$. 

In the context of normal distributions, there is a distinguished subgroup induced by affine transformations of the underlying sample space $\mathbb{R}$.
 More precisely, if we consider an affine map of the random variable $\zeta$ defined on $\mathbb{R}$ (i.e. $\zeta \mapsto a \zeta +b$), then under push-forward of measures, this induces a self-map of $\M_{\tt Normal}$. In fact, we can write out the embedding of the affine group into $PSL(2,\mathbb{R})$ explicitly as
 \begin{eqnarray*}
 Aff^+ \to PSL(2,\mathbb{R}) \\
 a x + b \mapsto  \left[
\begin{array}{cc}
\sqrt{a}  & \frac{b}{\sqrt{a}} \\
 0 &  \frac{1}{\sqrt{a}}  \\
\end{array}
\right] 
 \end{eqnarray*}
 
 
The phenomena that diffeomorphisms of the sample space induce isometries of the moduli space is more general, and holds for any parametrized family of probability distributions. 

\begin{theorem}[Chentsov \cite{cencov2000statistical}] \label{Chentsov's corollary}
The Fisher metric is invariant under the diffeomorphism group of the sample space $\mathcal{S}$.
\end{theorem}
 
 This result was originally proven by Chentsov\footnote{It is worth noting this theorem is actually a corollary of Chentsov's main result, which states that the Fisher metric is invariant under \emph{sufficient statistics}, which are statistics $T(Z)$ (for data $Z$) where the data processing inequality becomes an equality: \[ I\bigl(x ; T(Z)\bigr) = I(x ; Z).\] } when the sample space is finite, but it holds in much greater generality as well (see, e.g., Corollary 3.6 of \cite{ay2015information}). 
In view of this result, we say that an isometry of the Fisher metric is a \textit{Fisher-Chentsov isometry} if it is induced by map of the underlying random variable. It is worth noting that $\M_{\tt Normal}$ also has non-Fisher-Chentsov isometries, which correspond to Mobius transformations where $c$ in Equation \eqref{Mobius transformation} is non-zero.


\subsection{Constructing the statistical mirror pair}
\label{SMS construction}

With these preliminary calculations finished, we can now construct the statistical mirror pair. We have seen that the moduli space of univariate normal distributions is a hyperbolic manifold with dual Hessian structures, one induced by the natural parameters and the other induced by the sufficient statistics. Using these dual Hessian manifolds, we can construct two K\"ahler manifolds $\mathbb{M}_{\tt Normal}$ and $\mathbb{W}_{\tt Normal}$, which are said to be a \textit{statistical mirror pair}.

The primal manifold (denoted $\mathbb{M}_{\tt Normal}$),
is constructed on the tube domain\footnote{Observant readers will notice that this notation is reminiscent of tangent bundles. As we will see in Section \ref{Construction}, this is not coincidental.} $T \Omega \subset \mathbb{C}^2$, which is defined as
\[ T\Omega = \{ x^j + \sqrt{-1} y^j ~| ~ y^j \in \mathbb{R}^n, (x^1,x^2) \in \Omega \}   \subset \mathbb{C}^2. \]
Recall that $\Omega$ is a half space in $\mathbb{R}^2$, so $T\Omega$ is a half space in $\mathbb{C}^2$.
We use the log-partition function $\Phi$ to define a K\"ahler metric on  $\mathbb{M}_{\tt Normal}$. To do so, we define the lift of $\Phi$, denoted $\Phi^h$, as
\[\Phi^h(x+\sqrt{-1}y) = \Phi(x) \]
and consider the K\"ahler metric
\[\omega_{\mathbb{M}_{\tt Normal}}=\frac{\sqrt{-1}}{2} \frac{\partial^2 \Phi^h}{\partial x^i \partial  x^j} dz^i \wedge d \overline{z}^j.\]

In other words, the K\"ahler potential is simply $\Phi$, extended so that it is constant with respect to the imaginary directions $y$. Since $\Phi$ is strongly convex, $\Phi^h$ is strongly pluri-subharmonic, so induces a K\"ahler metric on $\mathbb{M}_{\tt Normal}.$

To build the mirror manifold $\mathbb{W}_{\tt Normal}$, we consider the tube domain over the parabola $\Omega^\ast$, i.e.
$T\Omega^\ast \subset \mathbb{C}^2$ with holomorphic coordinates
$\{w^j=u^j+\sqrt{-1} v^j\}_{j=1}^n$ for $(u^1,u^2) \in \Omega^\ast$, the domain of the dual variable. 
 Its K\"ahler metric is induced by lifting the conjugate potential $\Phi^\ast$ to $T \Omega^\ast$ and defining
\[\omega_{\mathbb{W}_{\tt Normal}}=\frac{\sqrt{-1}}{2} \frac{\partial^2 (\Phi^\ast)^h}{\partial u_i \partial u_j}  dw^i \wedge d \overline{w}^j. \]

We will provide a general definition for statistical mirror symmetry in Section \ref{Construction} (Definition \ref{SMS-abstract}). For now, we simply define the K\"ahler manifolds $\mathbb{M}_{\tt Normal}$ and $\mathbb{W}_{\tt Normal}$ to be a statistical mirror pair.

\subsection{The geometry of $\mathbb{M}_{\tt Normal}$ and $\mathbb{W}_{\tt Normal}$}


The spaces $\mathbb{M}_{\tt Normal}$ and $\mathbb{W}_{\tt Normal}$ have been studied extensively in the literature, and a summary of their geometric properties is provided in Table 1. The space $\mathbb{W}_{\tt Normal}$ is more commonly known as a \emph{Siegel domain of the second type} and has been studied in the context of automorphic forms and Abelian varieties. $\mathbb{M}_{\tt Normal}$ is known as the \emph{ Siegel-Jacobi space}, and has also been studied in the context of automorphic forms.\footnote{The associated K\"ahler metric is occasionally referred to as the K\"ahler-Brandt metric (see, e.g., \cite{molitor2014gaussian}), although we will not use this terminology.} The first observation that we make is that $\mathbb{M}_{\tt Normal}$ and $\mathbb{W}_{\tt Normal}$ are distinct from the viewpoint of complex geometry.

{\renewcommand{\arraystretch}{1.4}
\begin{table}[]
    \centering

\label{Table_normal_mirror}
\begin{tabular}[ht!]{|c | c| c|}
\hline
 & $\mathbb{M}_{\tt Normal}$  &   $\mathbb{W}_{\tt Normal}$   \\
\hline
also known as & the Siegel-Jacobi space & a Siegel domain (m=n=1) \\
\hline
base coordinates &   $ (x^1, x^2) = \left(\frac{\mu}{\sigma^2} , \, -\frac{1}{2 \sigma^2} \right)$  & $(u_1, u_2) = (\mu,\, \mu^2 + \sigma^2) $   \\ 
\hline
base domain & $\Omega = \{ (x^1, x^2) \in \R^2: x^2 \geq 0 \}$ & $\Omega^\ast = \{ (u_1, u_2) \in \R^2:  u_1 \cdot u_1 < u_2 \}$   \\
\hline
holomorphic coordinates &  $\{z^j=x^j+ \sqrt{-1} \, y^j\}_{j=1, 2} \in T\Omega$ &   
$ \{w_j=u_j+ \sqrt{-1} \, v_j\}_{j=1,2} \in T\Omega^\ast$   \\
\hline
biholomorphic to & the half-space in $\mathbb{C}^2$ & the unit ball in $\mathbb{C}^2$ \\
\hline
convex potential & $\Phi(x) = \frac{x^1 \cdot x^1}{4 (-x^2)} -\frac{1}{2} \log \left(  \frac{-x^2}{\pi} \right) $ & $\Phi^\ast(u) = -\frac{1}{2} \log (u_2 - u_1 \cdot u_1) $ \\
\hline
K\"ahler potential & $ \Phi^h(x+\sqrt{-1} \, y) = \Phi(x) $ & $ (\Phi^\ast)^h (u+\sqrt{-1} \, v) = \Phi^\ast(u) $ \\
\hline
K\"ahler metric & $\frac{\sqrt{-1}}{2} \frac{\partial^2 \Phi^h}{\partial x^i \partial  x^j} dz^i \wedge d \overline{z}^j $ & 
$ \frac{\sqrt{-1}}{2} \frac{\partial^2 (\Phi^\ast)^h}{\partial u_i \partial u_j}  dw^i \wedge d \overline{w}^j $  \\
\hline
metric known as & the K\"ahler-Brandt metric &  Bergman metric  \\
\hline
symmetry group & $SL(2,\R) \ltimes \R^2$ &  $PU(2,1)$ \\
\hline
geometric properties & cscK expanding K\"ahler-Ricci soliton & complex hyperbolic space form \\
\hline
encountered in &  automorphic forms & automorphic forms/Abelian varieties \\
\hline
\end{tabular}
\caption{Summary of the mirror pair $\mathbb{M}_{normal}$ and $\mathbb{W}_{normal}$ }
\end{table}
}


\begin{obs*}
$\mathbb{M}_{\tt Normal}$ and $\mathbb{W}_{\tt Normal}$ are \textit{not} biholomorphic.
\end{obs*}

\begin{proof}
It is a classic fact in complex analysis that this Siegel domain (i.e., $\mathbb{W}_{\tt Normal}$) is biholomorphic to the unit ball $\mathbb{B}^2 \subset \mathbb{C}^2$, by the map:
\[ F(w_1,w_2) = \left( \frac{2w_1}{\sqrt{-1} + w_2}, \frac{1+\sqrt{-1} w_2}{\sqrt{-1}+w_2} \right) . \]
For the sake of contradiction, suppose there were a biholomorphism $\phi:\mathbb{M}_{\tt Normal} \to \mathbb{B}^2$. Then the function $z^2 \to \phi(\sqrt{-1}, z^2)$ would be entire on $\mathbb{C}$ and bounded in each of its coordinates, which is impossible.
\end{proof}

Furthermore, $\mathbb{M}_{\tt Normal}$ and $\mathbb{W}_{\tt Normal}$ have different automorphism groups, which shows that the duality between them does not preserve the automorphism group. Among complex domains, the ball has the largest possible automorphism group \cite{isaev2004characterization}, so the automorphism group of $\mathbb{M}_{\tt Normal}$ is smaller. These observations should not be particularly surprising; the Legendre transformation can have a complicated effect on the domain of a convex function, which corresponds to deforming the complex structure of its tube domain\footnote{In coordinate-invariant language, Legendre duality does not preserve the affine structure of an affine manifold.} (see \cite{shimizu2000classification} for a details on the automorphism group of tube domains).

With Calabi-Yau manifolds, mirror symmetry can change the topology (it acts to rotate the Hodge diamond \cite{kontsevich1995homological}), so mirror pairs are generally not diffeomorphic. However, $\mathbb{M}_{\tt Normal}$ and $\mathbb{W}_{\tt Normal}$ are both tube domains over a convex subset of $\mathbb{R}^2$, and thus are diffeomorphic.\footnote{More conceptually, statistical mirror pairs are defined on the tangent bundle $T \mathcal{M}$ of an affine manifold $\mathcal{M}$, so they are always diffeomorphic (see Definition \ref{SMS-abstract}).}

 Finally, we remark that $\mathbb{M}_{\tt Normal}$ and $\mathbb{W}_{\tt Normal}$ are both domains of holomorphy, as they are tube domains whose base is convex \cite{yang1982automorphism}. This fact holds for any statistical mirror pair induced by an exponential family, since the domain of the natural parameters and expectation parameters are always convex.


\begin{obs*}
For any exponential family $ {\tt S}$, the K\"ahler mirror pair $\mathbb{M}_{\tt S}$ and $\mathbb{W}_{\tt S}$ are domains of holomorphy.
\end{obs*}


\subsubsection{Symmetries of $\mathbb{M}_{\tt Normal}$}

We now study the geometry of $\mathbb{M}_{\tt Normal}$ in more detail.  For a more complete reference on this space, we refer the reader to the following papers \cite{yang2007invariant,yang2016differential,yang2013sectional}. The first two of these study the geometry of this space from the perspective of automorphic forms and spectral analysis. The third computes the curvature tensor explicitly. We also recommend the work of Molitor \cite{molitor2014gaussian}, which studied this space from the perspective of mathematical physics.

 We start by considering the symmetries of $\mathbb{M}_{\tt Normal}$. The group of holomorphic isometries is the affine symplectic group $SL(2,\R) \ltimes \R^2$ \cite{molitor2014gaussian}. As a result, the group of holomorphic symmetries is (real) $5$-dimensional, and acts transitively.

 Intuitively, it is reasonable that the symmetry group of this space is real five-dimensional. In particular, any two-dimensional tube domain has a natural $\mathbb{R}^2$-action induced by translation in the fibers and the symmetry group of the half-plane is three-dimensional.  
 As we shall see shortly, it is not always the case that the holomorphic symmetry groups of tube domains are the semi-direct product of the affine symmetries group on the base and the $\mathbb{R}^n$-symmetry induced by translation in the fibers. However, this is indeed the case for $\mathbb{M}_{\tt Normal}$.

It is worth noting that $\mathbb{M}_{\tt Normal}$ is often studied in the context of the Jacobi group $G^J$, which is defined as
\[G^J := SL(2, \mathbb{R}) \ltimes \operatorname{Heis}(\mathbb{R}), \] 
where $\operatorname{Heis}(\mathbb{R})$ is the Heisenberg group (see \cite{yang2007invariant} for details). There is a natural action of this group on $\mathbb{M}_{\tt Normal}$, under which the metric is invariant. However, this action is not faithful, which is why the isometry group is smaller.\footnote{The affine symplectic group is not a subgroup of the Jacobi group. The latter is a central extension of the former.} Nonetheless, the Jacobi group plays an important role in the study of modular forms on this space, which is why we introduce it now.

From the fact that $\mathbb{M}_{\tt Normal}$ is homogeneous, it immediately follows that $\mathbb{M}_{\tt Normal}$ is complete. However, there is a more general result, which was proven by Molitor.
\begin{proposition*}[Proposition 2.15  \cite{molitor2014gaussian}]
A K\"ahler manifold defined on the tangent bundle $T \M$ of a Hessian manifold $\M$ is complete if and only if the underlying Hessian manifold $(\M,g,D)$ is complete.
\end{proposition*}


The horizontal submanifolds \[ \M_{a,b} = \{(z^1,z^2) ~|~ y^1=a, y^2=b \} \]
induce a totally geodesic foliation of $\mathbb{M}_{\tt Normal}$. Since each of these submanifolds is hyperbolic (and the transition maps to the standard half-plane model are given by Equation \eqref{Transition map from natural to mu-sigma}), we can calculate horizontal geodesics and distances explicitly.  The vertical submanifolds (where the $x$-coordinates are fixed) are Lagrange submanifolds, but are not totally geodesic and so we are not aware of a similar characterization for geodesics whose vertical displacement is non-zero.

\subsubsection{The Ricci and scalar curvature of $\mathbb{M}_{\tt Normal}$}

We now study some curvature properties of $\mathbb{M}_{\tt Normal}$. On a K\"ahler manifold, the Ricci curvature can be expressed as a $(1,1)$-form, which can be computed as
\begin{eqnarray*} Ric_{i\bar j} &=& - \sqrt{-1} \frac{\partial}{\partial z^i} \frac{\partial}{\partial \bar z^j} \log (\det[g]).
\end{eqnarray*}

For $\mathbb{M}_{\tt Normal}$, the K\"ahler form is given by $\omega = \partial \bar \partial \Phi^h$. Computing the Ricci curvature in the $(x_1, y_1, x_2, y_2)$ coordinates, we find that it is given by

 \[ Ric = \left[
\begin{array}{cccc}
0  & 0 & 0 & 0 \\
 0 & -\frac{3}{2 x_2^2} & 0 & 0  \\
 0 & 0 & 0 & 0 \\
 0 & 0 & 0 & -\frac{3}{2 x_2^2}
\end{array}
\right] . \]

For $\mathbb{M}_{\tt Normal}$, the K\"ahler potential $\Phi^h$ is a function of the base (i.e. the $x$-coordinates) alone.
As such, we can calculate the Ricci tensor by computing
\[ \frac{\partial^2}{\partial x^i \partial x^j} \log \left( \det \left[ \frac{\partial^2}{\partial x^a \partial x^b} \Phi \right] \right). \]
This reduces the original $4\times 4$ matrix of Ricci components to a $2 \times 2$ matrix, which greatly simplifies the analysis.

From this computation, we can see that the Ricci curvature is non-positive. However, this metric is not K\"ahler-Einstein. It is worth emphasizing that the Ricci tensor of $\mathbb{M}_{\tt Normal}$ is \textit{not} the Ricci tensor of $\M_{\tt Normal}$ (which is a hyperbolic Riemann surface).
 Taking the trace of the Ricci curvature with the inverse metric, we also find that $\mathbb{M}_{\tt Normal}$ has constant scalar curvature $-6$.

\subsubsection{The bisectional and anti-bisectional curvatures}

In Riemannian geometry, it is often of interest to consider metrics whose sectional curvature is either positive or negative. Oftentimes, this sort of assumption is necessary to obtain analytic results.

In K\"ahler geometry, it is often too restrictive to assume that the sectional curvature has a sign. However, due to the additional structure and symmetries introduced by the complex structure (see Chapter 4 of \cite{ballmann2006lectures} for an introduction), it is possible to define other curvature conditions which play an important role in the analysis of K\"ahler manifolds. Below, we note some of the curvature properties of $\mathbb{M}_{\tt Normal}$. This list is not intended to be exhaustive, and it is likely that there are other interesting curvature properties of this metric, which may play an important role in its analysis.

\begin{enumerate}
    \item The holomorphic sectional curvature does not have a sign \cite{molitor2014gaussian}. At every point there are directions with both negative and positive holomorphic sectional curvature. Similarly, the orthogonal bisectional curvature does not have a sign.
    \item The \emph{orthogonal anti-bisectional curvature} is non-negative definite. 
     The orthogonal anti-bisectional curvature plays an important role in the regularity theory of optimal transport. In particular, if a convex potential $\Phi$ induces a K\"ahler metric with non-negative anti-bisectional curvature, the associated cost function $c(x,y) = \Phi(x-y)$ is \emph{weakly regular}. In other words, there are no local obstructions to establishing continuity for solutions to the Monge problem with respect to this cost.\footnote{More precisely, the orthogonal anti-bisectional curvature is proportional to the MTW tensor, and so non-negative orthogonal anti-bisectional curvature of the K\"ahler metric is equivalent to the cost function satisfying the MTW(0) condition.} For details on the orthogonal anti-bisectional curvature and its relationship with optimal transport, we refer the reader to 
     \cite{khan2020kahler}. We discuss this particular metric in Example 9.
\end{enumerate}

\subsubsection{The geometry of $\mathbb{W}_{\tt Normal}$}

We now turn our attention to the mirror K\"ahler manifold $\mathbb{W}_{\tt Normal}$. As we mentioned previously, this is a Siegel domain, which has been studied extensively in the context of automorphic forms. Let us note some important properties of this space.

\begin{enumerate}
    \item The K\"ahler metric is the Bergman metric. As a result, it is invariant under biholomorphisms and so its isometry group is the same as its automorphism group. 
    \item The space is homogeneous and, as such, complete.
\end{enumerate}

The first property has the following important consequence, which is a classical fact in several complex variables (see, e.g., \cite{shabat1992introduction}).

\begin{obs*}
The isometry group of $\mathbb{W}_{\tt Normal} $ is the projective unitary group $PU(2,1)$.
\end{obs*}

Since $PU(2,1)$ is real six-dimensional, the symmetry group is larger than the semi-direct product of translations on $\mathbb{R}^2$ and the symmetries of $\M_{\tt Normal}$. This occurs because there are automorphisms of $\mathbb{W}_{\tt Normal}$ which do not respect the tube domain structure.
 \begin{obs*}
   Statistical mirror symmetry need not preserve the isometry group of a K\"ahler manifold.
    \end{obs*}

It is worth noting that $\mathbb{W}_{\tt Normal}$ is the only strictly pseudo-convex tube domain in complex dimension two which is homogeneous\footnote{$\mathbb{M}_{\tt Normal}$ is homogeneous and pseudo-convex, but not strictly so.} (and its higher dimensional analogues are the only such domains in $\mathbb{C}^n$ \cite{ezhov2018classification}).

As before, the horizontal submanifolds \[ \mathbb{W}_{a,b} = \{(w_1,w_2) ~|~ v_1=a, v_2=b \} \]
induce a totally geodesic foliation of the $\mathbb{W}_{\tt Normal}$. Furthermore, since $\mathbb{W}_{\tt Normal}$ is a complex space form, it is possible to calculate distances and geodesics explicitly (see \cite{sandler1996distance} for details). 

\subsubsection{Curvature properties of $\mathbb{W}_{\tt Normal}$}

\label{Curvature properties of W normal}
Similarly to $\mathbb{M}_{\tt Normal}$, we can understand the geometry of $\mathbb{W}_{\tt Normal}$ by investigating its curvature.

\begin{enumerate}
    \item The metric has constant negative holomorphic sectional curvature.\footnote{There is a general theorem due to Shima which shows that if the a K\"ahler manifold constructed in this fashion has constant holomorphic sectional curvature, the underlying Riemannian manifold $\M$ must have constant sectional curvature \cite{shima1995hessian}. This gives an alternate proof that the space $\M_{\tt Normal}$ is hyperbolic.}

    \item Since $\mathbb{W}_{\tt Normal}$ has constant holomorphic sectional curvature, it is necessarily K\"ahler-Einstein and has constant negative scalar curvature. This space is distinguished in that it is the unique strictly pseudo-convex domain for which the Bergman metric is K\"ahler-Einstein \cite{fu1997strictly,huang2016remark}.
    \item The orthogonal anti-bisectional curvature vanishes and the space has negative cost-curvature (see \cite{khan2020krflow} for a definition of cost-curvature).
\end{enumerate}

By contrasting the geometry of $\mathbb{W}_{\tt Normal}$ with that of $\mathbb{M}_{\tt Normal}$, we find the following fact.

\begin{obs*}
    The dual of a of K\"ahler metric with constant holomorphic sectional curvature need not be K\"ahler-Einstein.
\end{obs*}

As we will see in Subsection \ref{negtrisection}, the statistical mirror pair of a K\"ahler manifold with constant holomorphic sectional curvature need not even have constant scalar curvature. Recently, Maeta\footnote{In his paper, the results are phrased in terms of Hesse-Einstein and Hesse solitons, but this statement is equivalent to this claim.} showed that the statistical mirror of a K\"ahler-Einstein space is a K\"ahler-Ricci soliton (see Corollary 5.3 of \cite{maeta2021self}).

This phenomena contrasts with mirror symmetry for compact Calabi-Yau manifolds, where the mirror pairs both have vanishing Ricci curvature. From a more conceptual perspective, the Legendre transform inverts the Hessian matrix of a potential. For a semi-flat K\"ahler metric, when the determinant of the Hessian is constant (which corresponds to Ricci flatness in the compact case), the determinant of its dual is also constant. However, for K\"ahler-Einstein metrics with non-zero scalar curvature, the determinant of the Hessian matrix will not be constant, and will instead satisfy a different Monge-Amp\`ere equation, which is not preserved by Legendre duality.\footnote{However, there are some simple examples where the dual of a K\"ahler-Einstein metric is again K\"ahler-Einstein. For instance, consider $\mathbb{H}= \{z = x + \sqrt{-1} y ~|~ y>0 \}$ with the K\"ahler potential $f(z) = -\log(y)$.}

\subsection{K\"ahler-Ricci flow, conjugate flow, and statistical mirror symmetry}

Thus far, the geometric properties of $\mathbb{M}_{\tt Normal}$ and $\mathbb{W}_{\tt Normal}$ we have discussed were previously known in the literature (although the duality between these spaces was not studied in depth). To motivate the study of statistical mirror symmetry, it is natural to ask for new insights into the geometry of these spaces. In this subsection, we do so by considering the relationship between this duality and K\"ahler-Ricci flow. The main result of this subsection is the following.

\begin{proposition}
The spaces $\mathbb{M}_{\tt Normal}$ and $\mathbb{W}_{\tt Normal}$ are K\"ahler-Ricci solitons\footnote{$\mathbb{W}_{\tt Normal}$ is a complex space form, so the content of this proposition is that $\mathbb{M}_{\tt Normal}$ is a soliton and that the coupling persists under the flow.} which are immortally coupled under K\"ahler-Ricci flow.
\end{proposition}

Before we discuss this result and some of its consequences, we first provide a brief background on K\"ahler-Ricci flow. The Ricci flow was first introduced by Hamilton \cite{hamilton1982three} and evolves a Riemannian metric by its Ricci curvature:
\[ \frac{\pt}{\pt t} g_{ij} = -2 Ric_{ij}. \]
 This flow plays a central role in modern geometric analysis, largely due to its role in the proof of the the Poincar\'{e} and Geometrization conjectures \cite{perelman2002entropy,perelman2003ricci}.
 If one starts with a K\"ahler metric, the evolving metrics will remain K\"ahler (with the same complex structure), and the resulting flow is called the K\"ahler-Ricci flow \cite{cao1985deformation}. This flow has also been used to prove many results in complex geometry (see, e.g., \cite{shi1997ricci,mok1988uniformization}).

Tube domains form a natural class of metrics where the K\"ahler-Ricci flow is particularly simple to study. In particular, the K\"ahler-Ricci flow induces a flow of the underlying Hessian metric, whereby the Hessian potential $\Phi$ evolves via the parabolic Monge-Amp\`ere equation\footnote{The flow on the underlying Hessian manifold is known as Hesse-Koszul flow, and was studied by Mirghafouri and Malek \cite{mirghafouri2017long}.}
\begin{equation} \label{HesseKoszulflow}
    \frac{\partial}{\partial t} \Phi = \log \left ( \det \left[ D^2 \Phi \right] \right ).
\end{equation}

\subsubsection{K\"ahler-Ricci flow on $\mathbb{M}_{\tt Normal}$ and $\mathbb{W}_{\tt Normal}$}

It is possible to compute the K\"ahler-Ricci flow on $\mathbb{M}_{\tt Normal}$ explicitly.
To do so, we consider the one-parameter family of potentials (indexed by $t$):
\[ \Phi_t = -\frac{x_1^2}{x_2}- t \log(-x_2) \]
defined on the half space $\mathbb{H} = \{ (x_1,x_2)~|~ x_2<0 \}$. 
 We then define a family of K\"ahler metrics\footnote{The K\"ahler metric corresponding to univariate normal distributions (in their natural parameters) corresponds to $\mathbb{M}_{1/2}$. Here, we can ignore the constant term $\frac{1}{2}\log(\pi)$ since it does not affect the metric.} on the tube domain $\mathbb{H} \times \sqrt{-1} \mathbb{R}^2$ by lifting the potentials $\Phi_t$:  
\[ \omega_t = \sqrt{-1}\partial \bar \partial  \Phi^h_t.\]

It is straightforward to verify that the potentials $\Phi_t$ solve Equation \eqref{HesseKoszulflow}, and so induce solutions to the K\"ahler-Ricci flow.
In fact, the space $\mathbb{M}_{\tt Normal}$ is an expanding K\"ahler-Ricci soliton. Upon rescaling, the flow acts by dilating the $z_1$ coordinate, but does not alter the geometry.\footnote{ This result gives a simple explanation to a fact about these spaces which was originally observed by Yang et. al. \cite{yang2013sectional}. They studied a two-parameter family of invariant metrics on this space (which they denote $\mathbb{H}_{1,1}, ds^2_{1,1,A,B}$) and found that the scalar curvature does not depend on the $B$ parameter. In our language, changing the $B$ parameter acts to rescale the $z_1$-coordinate, so the invariance of the scalar curvature follows from the fact that $\mathbb{M}_{\tt Normal}$ is a K\"ahler-Ricci soliton with constant scalar curvature.}

\subsubsection{The Conjugate flow}

Given a solution to the K\"ahler-Ricci flow on a tube domain (or more generally the tangent bundle of a Hessian manifold), we can  use statistical mirror symmetry to induce a conjugate flow on $\mathbb{W}$.\footnote{Here, we drop the subscripts \texttt{ Normal} to highlight that the metrics are not constant.}

To construct this flow, we consider Kahler-Ricci flow on $\mathbb{M}$ and its associated flow on the Hessian manifold $\M$. By taking the Legendre transformation at each time $t$, we obtain a dual Hessian manifold defined on a (time-dependent) domain $\Omega^\ast_t$. We can then consider a tube domain $ T \Omega^\ast_t$ and construct a K\"ahler metric by lifting the potential $\Phi^\ast_t$. 

This flow was previously studied by Fei and Picard  \cite{fei2019anomaly} as a model to understand the relationship between T-duality and Hermitian curvature flows.\footnote{The main focus of their work is actually the ``anomaly flow," which deforms conformally balanced Hermitian metrics. However, there are some deep connections between this flow and K\"ahler-Ricci flow \cite{fei2019unification}.}
In general, the domain $\Omega^\ast_t$ will evolve with $t$. From the viewpoint of statistical mirror symmetry, this shows that the \textit{complex structure} of $(T \Omega^\ast_t, \Phi^\ast_t)$ can change along the conjugate flow. On the other hand, the \emph{symplectic structure} remains constant under conjugate flow (which is one manifestation of the duality exchanging complex with symplectic geometry). As such, conjugate flow changes both the metric and the complex structure of $\mathbb{W}$, in such a way as to leave the symplectic structure fixed.
 
In the case of the statistical manifold of normal distributions, we can write out this flow explicitly by considering the one-parameter family of potentials $\Phi_t$ and computing the Legendre transform at each time. Doing so, we find that
\begin{equation}
    \Phi_t^\ast = -t-t \log \left( \frac{u_2-u_1^2}{2t} \right)
\end{equation} 
which is defined on the domain \[ \Omega^\ast_t = \{ (u_1,u_2) ~|~  u_2-u_1^2 \geq 0 \},\] (which happens to be independent of the time). As such, the conjugate flow on $\mathbb{W}$ is the family of K\"ahler metrics on  $T \Omega^\ast_t$ where the metric is given by $\omega^\ast_t = \sqrt{-1} \partial \bar \partial \Phi_t^\ast$.

This flow is defined for positive time, and for all time the associated K\"ahler manifold is complex hyperbolic, so this flow expands the metric uniformly. Furthermore, since $\mathbb{W}$ is a complex hyperbolic space form, we find the following.

\begin{proposition}
K\"ahler-Ricci flow on $\mathbb{M}_{\tt Normal}$ is identical to conjugate flow on $\mathbb{W}_{\tt Normal}$.
\end{proposition}

To explain why this is noteworthy, let us define conjugate flow more precisely. We say that a Hessian manifold (or its tube domain) evolves via conjugate flow if its Hessian metric evolves via the equation
\begin{equation} \label{Conjugate flow}
\frac{ \partial}{\partial t} g_{j k}=\frac{1}{2} \frac{\partial^{2}}{\partial x_{j} \partial x_{k}} \log \operatorname{det}\left(g\right)-\frac{1}{2} \sum_{q} \frac{\partial}{\partial x_{q}} \log \operatorname{det}\left(g\right) \frac{\partial}{\partial x_{j}} g_{q k}
\end{equation}
Under this flow, the first term is the same as for K\"ahler-Ricci flow and the second is a ``correction term" which deforms the affine structure of the Hessian manifold (or equivalently the complex structure of its tube domain). From this, we can see that in general conjugate flow and K\"ahler-Ricci flow are not the same.






\subsubsection{Preserved curvature inequalities}

One of the main motivations in finding K\"ahler-Ricci solitons is that they can be used to derive strong inequalities about the flow (see, e.g. \cite{hamilton1986four, cao1992harnack}) and find preserved curvature conditions. Since $\mathbb{M}_{\tt Normal}$ is a soliton with non-negative orthogonal anti-bisectional curvature and $\mathbb{W}_{\tt Normal}$ is a soliton with negative cost curvature, it is natural to ask whether either of these properties are preserved by K\"ahler-Ricci flow. In recent work of the first named author and F. Zheng, we studied these questions and showed that non-positive anti-bisectional curvature (as with $\mathbb{W}_{\tt Normal}$) is preserved.

\begin{theorem}[K-Zheng `20]
\label{Negative antibisectional curvature is preserved by KR flow}
Let $(T \Omega, \omega_t)$ be a family of complete K\"ahler metrics with bounded curvature which solve the K\"ahler-Ricci flow. Suppose that the initial metric $\omega_0$ has non-positive anti-bisectional curvature. Then for all positive time, anti-bisectional curvature remains non-positive.
\end{theorem}

We furthermore show that for complex surfaces, non-negative orthogonal anti-bisectional curvature (as with $\mathbb{M}_{\tt Normal}$) is preserved.

\begin{theorem}[K-Zheng `20] \label{AntibisectionalcurvatureRicciflow}
Suppose that $ \Omega \subset \mathbb{R}^2 $ is a convex domain and $\Phi :\Omega \to \mathbb{R}$ is a strongly convex function so that the associated K\"ahler manifold $(T \Omega, \omega_0)$ 
\begin{enumerate}
    \item is complete,
    \item has bounded curvature, and
    \item has non-negative orthogonal anti-bisectional curvature.
\end{enumerate}
 Then the orthogonal anti-bisectional curvature remains non-negative along K\"ahler-Ricci flow. 
\end{theorem}

Using Legendre duality, we can translate these results into the dual setting to find conditions which are preserved under conjugate flow. 

\begin{theorem} \label{Conjugate flow preserved curvature}
Let $\Phi(x,t)$ be a convex potential on a (time-dependent) tube domain that evolves via conjugate flow. Suppose that for all covectors $u = \sum_i u_i dx^i$ and $v= \sum_i v_i dx^i$, the potential $\Phi(x,0)$ satisfies the inequality
\begin{equation} \label{mirror negative antibisectional}
    \sum_{i,j,k,\ell} \mathfrak{W}^{ijk \ell} u_i u_j v_k v_\ell \leq 0,
\end{equation}
where $\mathfrak{W}^{ijk \ell}$ is defined to be the quantity\footnote{In the formula, we have written $\mathfrak{W}$ in this way so that the first two terms have corresponding terms in the anti-bisectional curvature and the final term is a ``correction" term.}
\begin{eqnarray*} \label{Mirror anti-bisectional}
   \mathfrak{W}^{ijk \ell}  =  \mathlarger{\mathlarger{\mathlarger{\sum}}}_{\alpha,\beta,\gamma, \delta,\epsilon,\eta,\zeta} \left ( \begin{aligned} &-\Phi^{i \alpha} \Phi^{j \beta} \Phi^{\gamma k} \Phi^{\delta \ell} \Phi_{\alpha \beta \gamma \delta}\\
  & + \Phi^{i \alpha} \Phi^{j \beta} \Phi_{\alpha \beta \gamma} \Phi^{k \epsilon} \Phi^{\ell \zeta} \Phi^{\gamma \eta } \Phi_{\epsilon \zeta \eta} \\
& - \frac{\partial}{\partial x^\delta} \left( \Phi^{i \alpha} \Phi^{j \beta} \Phi^{\gamma k} \right) \Phi^{\delta \ell} \Phi_{\alpha \beta \gamma}.
   \end{aligned} \right).
   \end{eqnarray*}
   
   Then this inequality persists for all time along the flow.
\end{theorem}

A calculation shows that $\mathfrak{W}$ is the anti-bisectional curvature of the dual manifold $\mathbb{M}$ and thus Equation \ref{mirror negative antibisectional} states that the anti-bisectional curvature of $\mathbb{M}$ is negative, so this theorem is a restatement of Theorem \ref{Negative antibisectional curvature is preserved by KR flow} in terms of conjugate flow. Ideally, one would want to use conjugate flow to prove new results about K\"ahler-Ricci flow, instead of the other way around. However, we will leave this question for future work and so leave Theorem \ref{Conjugate flow preserved curvature} as a proof of concept for this approach.

\subsection{Conjectural remarks} 
\label{conjectural remarks}
While both $\mathbb{M}_{\tt Normal}$ and $\mathbb{W}_{\tt Normal}$ are well known spaces, their duality has not drawn as much attention. In this section, we make some conjectural remarks to motivate further exploration on this topic.

One of the important mathematical successes of mirror symmetry has been to relate difficult questions in enumerative algebraic geometry to problems on their mirror side that can calculated more readily \cite{morrison1997making}. We hope that it might be possible to use the duality between $\mathbb{M}_{\tt Normal}$ and $\mathbb{W}_{\tt Normal}$ in a similar way.

 \subsubsection{Analytic number theory and the Saito-Kurokawa lift}

Both the Siegel-Jacobi space and the Siegel domain have been studied in the context of number theory, and it is seems likely that their duality might give additional insight.

An automorphic form is a function from a topological group $G$ to the complex numbers $\mathbb{C}$ which is invariant with respect to some discrete subgroup $\Gamma \subset G$. In this case, the relevant group on $\mathbb{M}_{\tt Normal}$ is the Jacobi group (see \cite{yang2015geometry} for details) and the associated group for $\mathbb{W}_{\tt Normal}$ is $Sp(2,\mathbb{R})$.

Since $\mathbb{M}_{\tt Normal}$ and $\mathbb{W}_{\tt Normal}$ are not biholomorphic (and these groups are not the same), the relationship between automorphic forms on these spaces will be somewhat complicated. 
However, the existence of the Saito-Kurokawa lift strongly suggests there is some connection. 

The Saito–Kurokawa lift takes modular forms on the hyperbolic half plane $\mathbb{H}$ to a distinguished class (a \emph{Spezialschar}) of Siegel modular forms on the Siegel half space. Its existence was conjectured independently by Saito and Kurokawa and proven by Andrianov \cite{andrianov1979modular}, Maass \cite{maass1979spezialschar} and Zagier \cite{zagier1979conjecture}. 
This lift can be understood as a composition of three mappings, with the last map being a map from certain Jacobi forms (i.e., automorphic forms on $\mathbb{M}_{\tt Normal}$) to certain automorphic forms (i.e., a Spezialschar) on the Siegel upper half-space. This space naturally contains $\mathbb{W}_{\tt Normal}$ (as the space of symmetric complex matrices whose imaginary parts are positive definite and whose diagonal elements are equal) and may also have a similar interpretation as a ``complexified exponential family." Keeping in mind that $\mathbb{W}_{\tt Normal}$ and $\mathbb{M}_{\tt Normal}$ are both models for the tangent bundle of the hyperbolic half plane, this picture suggests that there might be a natural way to understand the Saito-Kurokawa lift in terms of information geometry.

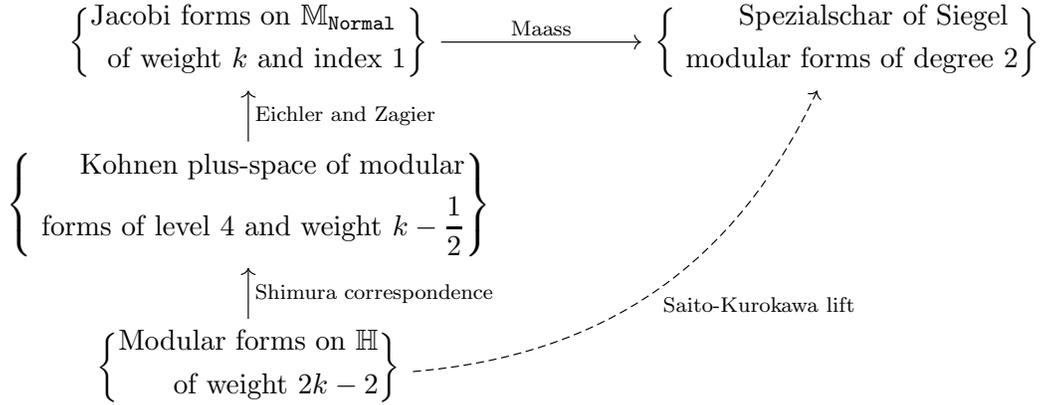
\begin{figure}
    \centering
    \begin{tikzcd}
\left \{\begin{aligned}
\textrm{Jacobi forms on $\mathbb{M}_{\tt Normal}$ } \\ \textrm{of weight $k$ and index 1}\end{aligned}  \right \} \arrow[rr, "\textrm{Maass}"]                  &  & \left \{ \begin{aligned} \textrm{ Spezialschar of Siegel  } \\ \textrm{ modular forms of degree 2}  \end{aligned} \right \} \\
\left \{ \begin{aligned} \textrm{ Kohnen plus-space of modular} \\ \textrm{\! \! forms of level 4 and weight $k-\frac{1}{2}$} \end{aligned} \right \} \arrow[u, "\textrm{Eichler and Zagier}"']                               &  &                                                                               \\
 \left \{ \begin{aligned}\textrm{Modular forms on } \mathbb{H} \\ \textrm{ of weight } 2k-2 \end{aligned} \right  \} \arrow[u, "\textrm{Shimura correspondence}"'] \arrow[rruu, "\textrm{Saito-Kurokawa lift}"', dashed, bend right] &  &                                                                              
\end{tikzcd}

    \caption{The Saito-Kurokawa lift}
    \label{Saito-Kurokawa lift}
\end{figure}

 \begin{question}
Is there a way to understand the Saito-Kurokawa lift using the duality between these two spaces? Are there any number-theoretic applications for this?
\end{question}

If this is indeed the case, it raises the following question.

\begin{question}
Can this type of duality be used more generally in the study of automorphic forms?
\end{question}

For the latter question, the natural spaces to consider are homogeneous tube domains with an invariant metric (such as a K\"ahler-Einstein metric) and to compute the Legendre transform of the underlying potential on the base to find a dual space. After considering the Siegel half-space and the Siegel-Jacobi space, the natural next class of examples are tube domains over homogeneous convex cones \cite{rothaus1966construction}, in which case the dual space will be a tube domain over the dual cone.


\subsubsection{The Moduli of Abelian Varieties}

It is possible to interpret the Siegel half-space as the moduli space of principally polarized Abelian varieties (see Chapter 8 of \cite{birkenhake2013complex}). As such, one can consider a point in $\mathbb{W}_{\tt Normal}$ as corresponding to the Abelian variety
 \[T_M = \mathbb{C}^2 / \left(M \mathbb{Z}^2 + \mathbb{Z}^2 \right) \]
where $M$ is a complex matrix of the form
\[ M = \left[
\begin{array}{cc}
z_1  & z_2 \\
z_2 & z_1
\end{array}
\right]  \]
whose imaginary part is positive definite.

It is of interest whether there is some moduli space interpretation for the Siegel-Jacobi space $\mathbb{M}_{\tt Normal}$ as well. Naively, one may hope that $\mathbb{M}$ is the moduli of a class of \emph{dual} Abelian varieties, and that this duality provides the right notion of correspondence between these spaces.

\begin{question}
Does $\mathbb{M}_{\tt Normal}$ have an interpretation as a moduli space of dual Abelian varieties?
\end{question}

Furthermore, $\mathbb{M}_{\tt Normal}$ and $\mathbb{W}_{\tt Normal}$ are not merely complex varieties, but also K\"ahler manifolds, which means that they also have a natural notion of distance and curvature. It is also possible to write down Darboux coordinates for these spaces (see Section \ref{Kahler mirror pair construction}), so it is straightforward to compute volumes in these spaces.\footnote{Neither of these spaces have finite volume, so this does not give a natural way to discuss \emph{random} Abelian varieties.} As a result, it seems natural to ask whether these metrics are intrinsically meaningful for the moduli of Abelian varieties. 

\begin{question}
 For genus $g$ curves, the Teichm\"uller metric provides a natural metric on the moduli space, which provides insight into the geometry of such surfaces.
    Does the K\"ahler metric on $\mathbb{W}_{\tt Normal}$ provide an analogy of a ``Teichm\"uller metric" for a class of principally polarized Abelian varieties? In other words, can we interpret it as a canonical K\"ahler metric in order to induce the ``distance" between Abelian varieties? If so, is there a corresponding interpretation for the K\"ahler metric on $\mathbb{M}_{\tt Normal}$ in terms of dual Abelian varieties?
\end{question}

\subsection{Generalization to the multivariate normal distribution}

It is also possible to study the moduli space of multivariate normal distributions ({\tt MulNor}). However, if one considers multivariate normal distributions with arbitrary covariance matrices, the resulting statistical manifold is an exponential family, but the Fisher metric is not hyperbolic \cite{skovgaard1984riemannian}. Instead, a more natural generalization is to consider the moduli space of isotropic multivariate normal distributions, which are multivariate normal distributions whose covariance matrix is $\sigma \, Id_n$ (where $\sigma$ is a real parameter), which we denote by {\tt IsoMulNor,n}. The associated statistical manifold is isometric to hyperbolic $n$-space $\mathbb{H}^n$. Furthermore, it is possible to parametrize isotropic multivariate normal distributions as an exponential family to obtain a mirror pair of K\"ahler manifolds. Doing so, one establishes a duality between $\mathbb{W}_{\tt IsoMulNor, n}$, which is the Siegel domain of degree $n$ and $\mathbb{M}_{\tt IsoMulNor, n}$, which is the Siegel-Jacobi space $\left (\mathbb{H}_{1,n}, ds^2_{1, n,1,1} \right )$ (using the notation of Yang \cite{yang2007invariant}). As with the univariate case, the dual space $\mathbb{W}_{\tt IsoMulNor, n}$ is a complex space form. The geometry $\mathbb{M}_{\tt IsoMulNor,n}$ is more complicated, but this perspective allows us to immediately understand some of its properties. 

To do so, we first compute the associated log-partition function and its Legendre dual. Doing so, we find that (ignoring some additive constants)
\[ \Phi_{ \tt IsoMulNor,n}(x^1, \ldots, x^n) = \frac{1}{2}\left( \frac{1}{x^1} \sum_{i=2}^{n} x^i \cdot x^i -\log \left ( -x^1 \right) \right) \] and its dual potential is
\[ \Phi_{ \tt IsoMulNor,n}^\ast(u_1, \ldots, u_n) = -\frac{1}{2} - \frac{1}{2}\log \left ( 2 u_1 - \sum_{i=2}^{n} u_i \cdot u_i \right). \] From this calculation, we are able to compute the Ricci and scalar curvatures of these metrics.

\begin{proposition}
The Ricci potential  $\rho_{\tt IsoMulNor,n} = -\log \left( \det \left [ D^2 \Phi_{\tt IsoMulNor,n} \right ] \right) $ of $\mathbb{M}_{\tt IsoMulNor,n}$ satisfies 
\[ \rho = - (n+1) \log (-2 \cdot x^1). \]
As a result, the scalar curvature of $\mathbb{M}_{\tt IsoMulNor,n}$ is $-(n+1).$
\end{proposition}

From this, we can see that the Ricci curvature is non-positive and the metrics have constant scalar curvature. Using Yang's notation, this answers Question 3 of \cite{yang2019problems} for the spaces\footnote{Note that we have used $n$ in the second index to match the dimension of the multivariate normal distribution. However, Yang normally denotes this index by $m$.} $\left (\mathbb{H}_{1,n}, ds^2_{1, n,A,B} \right )$. It is also possible to compute the orthogonal anti-bisectional curvature of these spaces. Using computer algebra to simplify the expression, we find the following.

\begin{proposition}
For all $n$, the space $\mathbb{M}_{\tt IsoMulNor,n}$ has non-negative orthogonal anti-bisectional curvature.
\end{proposition}

As a final note, we pose as an open question the problem of finding exponential families whose tube domains correspond to the other Siegel-Jacobi spaces. Doing so would give a way to compute the K\"ahler potentials of these spaces explicitly and determine other properties of their geometry in a straightforward way.

\section{What is Statistical Mirror Symmetry?}
\label{Construction}


While the correspondence between the Siegel-Jacobi space and the Siegel half-plane is the archetypal example of statistical mirror symmetry, this duality can be defined for more general Hessian manifolds. In this section, we provide a precise definition for the concept and contrast it with mirror symmetry of Calabi-Yau manifolds. We will focus on the case when the mirrors are K\"ahler. It is possible to extend this idea to the non-K\"ahler case, and we will make some comments about this more construction at the end of this section.

\subsection{K\"ahler statistical mirror symmetry}
From a high level perspective, K\"ahler statistical mirror symmetry can be described quite simply. For any Hessian manifold $(\M,g, D)$, there is a dual Hessian manifold $(\M,g, D^\ast)$, where the affine structure is induced by dual coordinates and the metric is given by the Legendre dual of the original potential (in each affine chart). 
 Furthermore, for any Hessian manifold $(\M,g, D)$, it is possible to define a K\"ahler metric on its tangent bundle, by lifting the potential as we did in the previous section. This map from a Hessian manifold to its tangent bundle is known in the mathematical physics literature as the `r-map' \cite{alekseevsky2009geometric}.

\begin{definition}[Statistical mirror symmetry-K\"ahler case] \label{SMS-abstract}
Two K\"ahler manifolds $\mathbb{M}$ and $\mathbb{W}$ are a statistical mirror pair if they are constructed from the $r$-maps of dual Hessian structures $(\M,g, D)$ and $(\M,g, D^\ast)$.
\end{definition}

This definition of statistical mirror symmetry is concise, but requires a fair amount of background to be meaningful. It also does not provide any motivation for studying this notion, or give any insight into the geometry of the mirror pairs. As such, in the rest of this section we will construct such the two spaces more concretely. Before doing so, however, let us note on the connection between Definition \ref{SMS-abstract} and mirror symmetry for Calabi-Yau manifolds.

Recall that Calabi-Yau manifolds are K\"ahler manifolds whose first Chern class vanishes. By a celebrated result of Yau, such manifolds admit a metric of vanishing Ricci curvature. In the context of Hessian manifolds and their tangent bundles, the spaces $\mathbb{M}$ and $\mathbb{W}$ will be Ricci-flat whenever the Hessian potential satisfies the Monge-Amp\`ere equation
\begin{equation}
\label{MongeAmpere2}
     \det  D^2 \Phi = C.
\end{equation}
 In this case, if we quotient the fibers $\mathbb{M}$ by a co-compact lattice and quotient the fibers of $\mathbb{W}$ by its dual lattice,  we obtain Leung's construction of mirror Calabi-Yau manifolds in the semi-flat case \cite{leung2000mirror}. As such, statistical mirror symmetry is related to, but distinct from, mirror symmetry for semi-flat Calabi-Yau manifolds. 
 
 The differences between the two constructions are two-fold. First, we do not assume that our metrics are Ricci flat. In fact, the only compact Calabi-Yau manifolds (with their Ricci-flat metrics) which arise via pulling back solutions to Equation \ref{MongeAmpere2} over a compact affine base are complex tori \cite{ChengYau}. As such, the non-trivial examples of statistical mirror symmetry will not be Calabi-Yau. 
 
 Second, we do not compactify the fibers by quotienting by a co-compact lattice and its dual. There are several reasons why it preferable to consider the full tangent bundle in this context. First, for Hessian manifolds whose affine structures are non-trivial, it is generally not possible to quotient by a co-compact lattice in a consistent way. Second, even in the cases where the affine structure is trivial, quotienting the fibers obscures aspects of the correspondence between the primal and dual space. For instance, for the Siegel half-space and the Siegel-Jacobi space, quotienting the fibers seems to destroy any potential correspondence of the modular forms.

\subsubsection{Constructing K\"ahler statistical mirror pairs}
\label{Kahler mirror pair construction}

To construct statistical mirror pairs more concretely, it is helpful to provide a second definition for Hessian manifolds which is equivalent to Definition \ref{Hessian manifold-coordinate definition} and which makes the duality of Hessian manifolds much more explicit (at the expense of suppressing the role the ``Hessian" plays).

\begin{definition*}[Hessian  manifold -- Information geometric] \label{Hessian manifold connection definition}
 A Riemannian manifold $(\M,g)$ is said to be \textit{Hessian} if it admits dually flat connections. That is to say, it admits two flat (torsion- and curvature-free) connections $D$ and $D^*$ satisfying
\begin{equation} \label{conjugateconnection}
X(g(Y,Z)) = g(D_X Y, Z) + g(Y, D^*_X Z) 
\end{equation}
 for all vector fields $X,~ Y$, and  $Z$. Because of these dual flat connections, a Hessian manifold is often said to be \textit{dually flat}.
\end{definition*}

The duality between the connections $D$ and $D^\ast$ is the crux of statistical mirror symmetry. However, in order to view this as a correspondence between complex and symplectic manifolds, we consider this duality in terms of the tangent bundle.
 For any Riemannian manifold $(\M,g)$ together with an affine connection $D$ (not necessarily curvature-free nor torsion-free), it is possible to construct an almost-Hermitian structure on the tangent bundle $T \M$ known as the Sasaki metric (for details on this construction, we refer the reader to \cite{dombrowski1962geometry,satoh2007almost}).

For an arbitrary connection and metric, these structures are somewhat complicated to write out explicitly (see Section \ref{nonK_mirror} for the exact formulas). However, for a flat connection $D$, this construction simplifies greatly if we work in terms of the associated affine coordinates (in which the Christoffel symbols vanish). Namely, if we let $x=\{ x^1, \cdots, x^n \}$, be such a coordinate chart on $\M$, we consider the associated coordinates $\phi: T \M \to \mathbb{R}^{2n} $ given by
\[ \phi \left(\left. y^{i} \frac{\partial}{\partial x^{i}}\right|_{p}\right)=\left(x^{1}(p), \ldots, x^{n}(p), y^{1}, \ldots, y^{n}\right). \] 
We then consider the complex valued coordinates $\{ z^i = x^i + \sqrt{-1} y^i \}_{i=1}^n$, which form holomorphic coordinates on $T \M$. As a result, the tangent bundle is in fact a complex manifold, not simply almost complex.\footnote{The fact that the Sasaki metric is complex is a consequence of the connection $D$ being flat (see Proposition \ref{prop_dombrowski}).}  We define the Hermitian metric $h$ as
\[ h(\partial_{x^i},\partial_{x^j})= g(\partial_{x^i},\partial_{x^i}), \,\,\,\, h(\partial_{y^i},\partial_{y^j})= g(\partial_{x^i},\partial_{x^i}),\]
and 
\[h(\partial_{x^i},\partial_{y^j})=0.\]

When $(\M,g,D)$ is a Hessian manifold, it is possible to simplify this expression by lifting the Hessian potential $\Phi$ to $T\M$. As we did in the context of normal distributions, we define the horizontal lift
\[\Phi^h(z) = \Phi(Re(z))\]
and define the K\"ahler form
\[\omega_\Phi = \frac{\partial^2}{\partial z^i \partial \bar z^j} \Phi^h. \]

With this form, $T \M$ is a K\"ahler manifold and it is straightforward to show that the associated Hermitian metric is $h$. Furthermore, we can also write out Darboux coordinates for this space explicitly. To do so, we note that since $\M$ is a Hessian manifold, the dual connection $D^\ast$ is flat. As such, we can use it to induce coordinates $u$ on $\M$ which are bi-orthogonal to the $x$-coordinates (i.e. satisfy $g(x^i,u_j)=\delta^i_j$). In fact, we can write out these coordinates explicitly using the potential function as
\[u_i = D_{x^i} \Phi. \]

\begin{obs*}
The paired coordinates $\{ (u^i ,y^i)\}_{i=1}^n$ are Darboux coordinates on $T \M$.
\end{obs*}

To construct the dual K\"ahler manifold $\mathbb{W}$, we repeat this construction, but use the \textit{dual connection} $D^\ast$ instead of the primal one. Associated with the preferred coordinate of the dual connection is a convex function $\Phi^\ast$, which is the convex dual of $\Phi$. Furthermore, the primal $x$-coordinates can be obtained by differentiating $\Phi^\ast$:
$$
x^i = D^\ast_{u_i} \Phi^\ast .
$$

\subsubsection{Tube domains}

The previous discussion provides an explicit description of the K\"ahler geometry of the tangent bundle of a Hessian manifold. In this paper we will primarily focus on a special case, which is when the associated Hessian manifold is a domain in Euclidean space. If we consider a domain $\Omega$ in Euclidean space, we can produce a flat connection simply by differentiating in coordinates. In this case, the tangent bundle is simply the tube domain $T \Omega$, 
\[ T \Omega = \{ z = x+ \sqrt{-1}y \in \mathbb{C}^n ~|~ x \in \Omega, y \in \mathbb{R}^n \}. \]

 From a topological perspective, these spaces are simple (they are trivial when $\Omega$ is convex). Nonetheless, statistical mirror symmetry produces surprisingly rich structure from the perspective of complex geometry. 


When $\M$ is a Hessian manifold which is defined on a convex domain in Euclidean space,  we can describe the geometry of $\mathbb{M}$ explicitly.
Since $\M$ is a convex domain, the Riemannian metric is given by
\[ g = D^2 \Phi\] for some globally defined strongly-convex function $\Phi :\M \to \mathbb{R}$.
The K\"ahler potential on $\mathbb{M}=T\M$ is then given by the horizontal lift of the convex potential $\Phi$.

\begin{definition}[Horizontal Lift]
A function $\Phi^h: \mathbb{M} \to \mathbb{R}$ is the horizontal lift of a function $\Phi: \M \to \mathbb{R}$ if
\begin{equation}
    \Phi^h(x,y) = \Phi(x).
\end{equation}
\end{definition}

 In other words, the horizontal lift ignores the ``vertical" part of the holomorphic coordinates. Since $\Phi$ is strongly convex on $\M$, $\Phi^h$ is strongly pluri-subharmonic and so defines a K\"ahler metric.

We can also give an explicit characterization of mirror space $\mathbb{W}.$ We first compute the Legendre transform of $\Phi$:
\[ \Phi^\ast(u) = \sup_{x \in \Omega} \, \langle u,x \rangle - \Phi(x) , \]
which is defined on $\Omega^\ast = \{ u \in \R^n: \Phi^\ast (u) < \infty\}$. Then we construct the tube domain $T \Omega^\ast$, which will become $\mathbb{W}$. The K\"ahler potential on $\mathbb{W}$ is just the horizontal lift $(\Phi^\ast)^h$.

\subsection{Non-K\"ahler statistical mirror symmetry}
\label{nonK_mirror}

It is possible to construct a non-K\"ahler version of statistical mirror symmetry as well. This case was studied in depth in our previous work \cite{zhang2020statistical}, and we will discuss it briefly here. Readers who are primarily interested in the K\"ahler theory should feel welcome to skip this subsection.

At a fundamental level, statistical mirror symmetry can be understood in terms of duality for affine connections. We have already seen dual connections in the context of Hessian manifolds, but this duality can be defined for arbitrary affine connections.\footnote{The notion of conjugate connections is a central notion in information geometry and has been studied in depth (see,  e.g., \cite{fei2017interaction,calin2009generalizations,matsuzoe2010statistical}).}

\begin{definition} \label{dual connection}
Given a Riemannian manifold with an affine connection $D$, the dual connection $D^\ast$ is the unique connection which satisfies
 \[ 
Z(g(X,Y)) = g(D_Z X, Y) + g(X, D^\ast_Z Y) ,
\] 
for vector fields $X,Y,$ and $Z$.
\end{definition}

Let us recall some basic properties of conjugate connections.
\begin{enumerate}
    \item The dual connection $D^\ast$ is curvature-free if and only if the primal connection $D$ is curvature-free. 
    \item  $D$ and $D^\ast$ have the same torsion tensor if and only if the quantity
    \[ C^D :=D g \]  is a totally symmetric three-tensor.
    That is to say, if
 \[ C^D (X,Y,Z) \equiv (D_Z g)(X, Y) \]
for any three vector fields on $\M$. In this case, we say that the pair
    $(g, D)$ is \emph{Codazzi-coupled}.
    \item A straightforward computations shows that
$ C^{D^\ast} = - C^D, $ so $(g, D)$ is Codazzi-coupled if and only if $(g, D^\ast)$ is. 
\end{enumerate}

\subsubsection{General construction}

To define non-K\"ahler version statistical mirror symmetry, we consider an affine Riemannian manifold $(\M,g,D)$ with a flat connection $D$, but do not assume that $g$ is a Hessian metric (that is, $g$ cannot be written as $D^2 \Phi$ for some potential function $\Phi$). In this case, the connection $D^\ast$ will be curvature-free, but will have torsion.\footnote{Note that if $g$ is not Hessian,
$D$ and $g$ cannot be Codazzi coupled, because this would imply that the dual connection $D^{\ast}$ is flat.} It is worth noting that such manifolds are not statistical manifolds, since statistical manifolds are defined to be Riemannian manifold $(\M,g)$ with an affine connection $D$ such that both $D$ and $D^\ast$ are torsion-free \cite{lauritzen1987statistical}. Instead, these are examples of \emph{statistical manifold admitting torsion} \cite{kurose2007statistical}.

A manifold $(\M, g, D, D^\ast)$ where (i) $D$ is flat (both curvature- and torsion-free) and (ii) $D^\ast$ is curvature-free but has non-zero torsion is said to be \emph{partially flat} \cite{henmi2019statistical}.\footnote{In this paper, the authors used the convention that $D^\ast$ is the flat connection, not $D$. However, to be more consistent with the broader literature on mirror symmetry, we chose to flip this definition in our previous work.}

For any partially flat manifold, there are a pair of almost-Hermitian structures on the tangent space $T \M$, with the first being a Hermitian manifold $\mathbb{M}$ constructed from $D$ and the second an almost K\"ahler manifold $\mathbb{W}$ constructed from $D^\ast$. These spaces are said to be in mirror correspondence, and provide a non-K\"ahler generalization of statistical mirror symmetry. 
To construct these spaces, we again consider coordinates $x= \{ x^1, \cdots, x^n\}$ and the associated bundle coordinates
\[ \phi \left(\left. y^{i} \frac{\partial}{\partial x^{i}}\right|_{p}\right)=\left(x^{1}(p), \ldots, x^{n}(p), y^{1}, \ldots, y^{n}\right). \]

Given an affine connection $D$ on $\M$ and a point $(x,y) \in T\M$, we construct two lifts $X^H$ and $X^V$ of any vector $X \in T_{x} \M$, which are vectors in $T_{(x,y)} (T\M)$. These are called the the {\it horizontal} and {\it vertical} lifts, from $T_{x}\M$ to $T_{(x,y)}(T\M)$, and are defined via the formulas
\begin{equation}
 X^V  =  X^i\frac{\partial}{\partial y^i}  \textrm{ \qquad and  \qquad }  X^H  =  X^i\frac{\partial}{\partial x^i} 
    - X^i y^j\Gamma^k_{ij}\frac{\partial}{\partial y^k}.
  \label{vertical and horizontal} 
\end{equation}

In this formula, the $\Gamma^k_{ij}$'s are the Christoffel symbols of $D$
with respect to the $x$-coordinates. In other words, if we consider the vectors $e_i = \frac{\partial}{\partial x^i}$, the $\Gamma^k_{ij}$ terms satisfy the identity
\[ D_{e_i} e_j = \Gamma^k_{ij} e_k. \]

 These lifts define a splitting of the tangent space $T_{(x,y)}(T\M)$ as
\[  T_{(x,y)}(T\M)=\mathcal{H}_{(x,y)}(T\M) \oplus\mathcal{V}_{(x,y)} (T\M),
\]
where the spaces $\mathcal{H}$ and $\mathcal{V}$ are the images of the horizontal and vertical lift, respectively.
Note that the vertical lift (and thus the vertical space\footnote{More conceptually, the vertical space is the kernel of the pushforward of the projection map $ \pi: T \M \to \M$.}) is independent of the choice of connection but the horizontal lift depends on the connection. We can then use this splitting to define an almost complex structure $J_D$ and almost-Hermitian metric $G_D$. 
\begin{enumerate}
    \item The almost complex structure $J_D$ is defined by
\[
J_D(X^H) = X^V, \,\,\,\,\,\,\,\,\, J_D(X^V) = -X^H.  \]
\item The metric on $T \M$ is lifted from the metric $g$ on $\M$ via the formula
\begin{eqnarray*}
  G_D \left(X^V,
Y^V
  \right) 
  = G_D \left(X^H, 
  Y^H
  \right)
  & = & g(X,Y), \\
  G_D \left(X^V, 
  Y^H
  \right)
  & = & 0,
\end{eqnarray*}
\end{enumerate}

It follows immediately that the Riemannian metric $G_D$ and the almost complex structure $J_D$ are orthogonal, which implies that the space $(T\M, J_D, G_D)$ is an almost Hermitian manifold.

\begin{definition}(Non-K\"ahler statistical mirror symmetry)
\label{NonKahler SMS definition}
Given a partially flat manifold $(\M, g, D, D^\ast)$, the almost-Hermitian manifolds 
\begin{enumerate}[label=(\roman*)]
    \item $\mathbb{M}= (T\M, J_D, G_D)$, and
    \item $\mathbb{W}= (T\M, J_{D^\ast}, G_{D^\ast})$
\end{enumerate}
are said to be a statistical mirror pair.
\end{definition}

To explain why this can be understood as a complex-to-symplectic duality, we first observe that 
 $\mathbb{M}$ is a Hermitian manifold (i.e., $J_D$ is integrable).
\begin{proposition} [Dombrowski \cite{dombrowski1962geometry}]
\label{prop_dombrowski}
The almost complex structure $J_D$ on $T\M$ is integrable if and only $D$ is flat.
\end{proposition}

In fact, we can construct holomorphic coordinates in the same way we did previously; by considering affine coordinates $x$ associated to $D$ and considering the coordinates  \[z= \left \{ z^1,\ldots,z^n ~|~ z^i= x^i+\sqrt{-1} y^i \right \}.\] 

On the mirror side, the following proposition shows that $\mathbb{W}$ is symplectic.
 
\begin{proposition}[Satoh \cite{satoh2007almost}]
\label{prop_sato}
The fundamental form $\omega^\ast = G_{D^\ast}(J_{D^\ast} \cdot,  \cdot)$ is closed (i.e., satisfies $d\omega^\ast = 0$) if and only if $D$ is torsion-free. 
\end{proposition}
Furthermore, the symplectic form on $\mathbb{W}$ is the pull-back of the canonical symplectic form on the cotangent bundle $T^\ast \M$. 
However, since $D^\ast$ has non-vanishing torsion, these results show that $\mathbb{M}$ is non-K\"ahler (since $d \omega \neq 0 $) and that the complex structure on $\mathbb{W}$ is not integrable.


\subsubsection{Non-K\"ahler statistical mirrors for parametrized families}


For any exponential family, it is possible to use the natural and expectation parameters to construct a K\"ahler statistical mirror pair. However, exponential families are a special class of statistical manifolds; for more general parametrized families, it is generally not possible to find a K\"ahler mirror pair (see Section \ref{Existence of Hessian metrics} for details). As such, non-K\"ahler statistical mirror symmetry is useful to understand statistical manifolds which are not exponential families.

In most cases of interest, the parameter space of a family of probability distributions is an open domain in Euclidean space, so most statistical manifold can be induced with a partially-flat structure. For each global set of coordinates on the domain, there is an associated flat connection which can be used to construct a mirror pair (which will generally be non-K\"ahler). As such, in order to find a canonical geometry, it seems necessary to impose additional conditions on the geometry of $\mathbb{M}$ and $\mathbb{W}$. In \cite{zhang2020statistical}, we proposed one possibility, which is to find a flat connection on $\M$ so that $\mathbb{M}$ is balanced (i.e., satisfies $d \omega^{n-1}=0$). The appeal of this geometric condition is that associated system of PDEs is determined (i.e., has the same number of unknowns as equations) and that (conformal) balancing arises in the Strominger system, which has been studied in theoretical physics \cite{garcia2016lectures}. 
 However, it is quite possible that statistical considerations will motivate a different geometric condition entirely.

\begin{question}
Given a parametrized family of distributions which is not an exponential family (and thus has no natural parameters), is there a natural condition to impose on the parametrization so that the associated complex manifold $\mathbb{M}$ has some desirable statistical properties? 
\end{question}

\section{More examples of hyperbolic statistical mirrors}
\label{Hyperbolic}

In Section \ref{Normaldistributions}, we saw that the moduli space of univariate normal distributions induces a mirror correspondence between two K\"ahler manifolds, one of which was a complex space form and the other of which was a K\"ahler-Ricci soliton. It might initially be tempting to think of this as being the ``canonical" hyperbolic statistical manifold.
 However, it turns out that this example is not unique, in that there are many other examples of statistical mirrors whose underlying statistical manifold has constant negative sectional curvature. It is even possible to find other examples where one of the K\"ahler metrics is a space of constant negative holomorphic sectional curvature. In this section, we will provide another such example, which is induced from the moduli space of negative trinomial distributions.
 We will also provide one further example of a statistical manifold $\M$ where the underlying Hessian metric is hyperbolic, which is derived from the family of inverse Gaussian distributions.

\subsection{The geometry of the negative trinomial distribution (NegTri)}
\label{negtrisection}

One of the common distributions one encounters in statistics is the binomial distribution, which is the distribution of heads when one performs $k$ independent flips of a coin which is heads with probability $p$ and tails with probability $1-p$. Conversely, one might flip such a coin repeatedly until encountering the first tail. The number of heads that appear before the first tail will be distributed according to the \textit{negative binomial distribution}. The negative multinomial distribution generalizes this distribution by having multiple success events rather than a single one. For example, if one roles a die repeatedly until a one appears, the distribution of the numbers two through five will be given by a negative multinomial (in this case, a negative sextinomial). 
More precisely, we consider repeated independent draws of a fixed multinomial distribution where one event is considered a failure event. The negative multinomial is the distribution of the number of draws of each of the other events (i.e. the successes) before the first failure occurs.\footnote{This can be generalized to allow for a specified number of failures $r$, but for simplicity, we will only consider the case where $r=1$.} This distribution is supported on the discrete space $ \left( \mathbb{Z}_{ \geq 0} \right)^{n-1} $ and the probability mass function is
\[ f(k_1, \ldots, k_{n-1} ~|~ p_1, \ldots,  p_{n-1}) = \Gamma\!\left(\sum_{i=0}^{n-1}{k_i}\right) \prod_{i=1}^{n-1}{\frac{p_i^{k_i}}{k_i!}}. \]
Here, $\Gamma$ is the gamma function and the $p_i$ are the probability of each of the success events. The zero-th event is considered the failure event and has probability $p_0.$

For the sake of concreteness, we will restrict our attention to the negative trinomial (i.e., where there are two success events and one failure event). However, all of the calculations in this section can be extended to the negative multinomial. As we did with normal distributions, we consider the moduli space of negative trinomial distributions as a statistical manifold, and compute its Fisher-Rao metric.

$$ 
g = \frac{1}{\left(1- p_1 - p_2 \right)^2} \left[
\begin{array}{cc}
p_1 (1-p_2)   & p_1 p_2   \\
p_1 p_2  & p_2 (1-p_1) \\
\end{array}
\right] .
$$
A straightforward calculation shows that this metric has constant curvature $-1/2$, and so we see that $\M_{\tt NegTri}$ is hyperbolic. However, unlike the Gaussian family, this metric is not complete. As such, the space of negative trinomials is a proper subset of hyperbolic space.
 To determine the global geometry, it is helpful to reparametrize the family in terms of the parameters $s_1 = \sqrt{p_1}$ and $s_2= \sqrt{p_2}$. In these new coordinates, the Fisher metric for the negative trinomial is
\[ g_{i j}=\frac{\delta_{i j}}{2(1-s_{1}^{2}-s_{2}^{2})}+\frac{s_{i} s_{j}}{2\left(1-s_{1}^{2}-s_{2}^{2}\right)^{2}},
\]
which is precisely the Klein disk model for hyperbolic space restricted to the first quadrant. 
This gives a picture of the global geometry of $\M_{\tt NegTri}$, at least from the the perspective of Riemannian geometry. From this, we see that the metric is unbounded, and that the boundaries of the manifold occur where the probability of the failure event approaches one.

The symmetry group of $\M_{\tt NegTri}$ is $\mathbb{Z}_2$, and is generated by exchanging the probabilities of the first and second success event.\footnote{For the negative multinomial, the symmetry group is $S_{n-1}$ and is induced by permuting success events.} This is a Fisher-Chentsov isometry, as it is induced by a map of the underlying random variable. To show that there are no other isometries of $\M_{\tt NegTri}$, note that these are the only isometries of the Klein disk (with its hyperbolic metric) which preserve the first quadrant. 

\subsubsection{Negative trinomials as an exponential family}

The family of negative trinomial distributions is an exponential family. The natural parameters of this family are
\begin{equation}\label{negtrinaturalparameters}
    \theta^1= \log(p_1) , \,\,\,\,\,\,\,\, \theta^2 =  \log(p_2) ,
\end{equation}
which is defined on the domain
\begin{equation} \label{negtriprimaldomain}
    \Omega_{\tt NegTri} = \left\{ (\theta^1,\theta^2) \in \R^2 ~|~ \exp(\theta^1)+\exp(\theta^2) < 1 \right\}
\end{equation}
The potential (log-partition) function is
\begin{equation} \label{Negtrilogpartition}
    \Phi_{\tt NegTri}(\theta)= -\log \left( 1-e^{\theta^1}-e^{\theta^2} \right).
\end{equation}
The corresponding sufficient statistics are
\begin{equation}\label{negtrisufficientstatistics} \eta_1 = \frac{p_1}{1-p_1 - p_2}, \,\,\,\,\,\,\,\, \eta_1 = \frac{p_2}{1-p_1-p_2}, 
\end{equation}
which is defined on the set 
\begin{equation} \label{negtridualdomain}
    \Omega^\ast_{\tt NegTri} = \{ (\eta_1,\eta_2) ~|~ \eta_1,\eta_1 > 0 \} .
\end{equation}

In order to transition from the $\theta$ to $\eta$ coordinates, we use the formula $\eta = D_\theta \Phi$:
\begin{equation} \label{negtritransitionmap}
\eta_1 =\frac{e^{\theta^1 }}{1 -e^{\theta^1 }-e^{\theta^2 }} ,   \,\,\,\,\,\,\,\, \eta_2 =\frac{e^{\theta^2 }}{1-e^{\theta^1 }-e^{\theta^2}}.  \end{equation}

\subsubsection{The K\"ahler geometry of $\mathbb{M}_{\tt NegTri}$}

As with the normal family, we can use the natural parameters and sufficient statistics of the negative trinomial distribution to construct a pair of statistical mirrors. Let us first consider the primal manifold $\mathbb{M}_{\tt NegTri}$.

This space is constructed on the tube domain $T \Omega_{\tt NegTri}$, which was the domain of the natural parameters \eqref{negtrinaturalparameters}. To obtain a K\"ahler metric on $T \Omega_{\tt NegTri}$, we compute the horizontal lift of the log-partition function $\Phi_{\tt NegTri}(\theta)$ given by \eqref{Negtrilogpartition}. We enumerate some of the important properties of $\mathbb{M}_{\tt NegTri}$ below.

\begin{enumerate}
\item It is incomplete (since $\M_{\tt NegTri}$ is incomplete).
\item It has constant negative holomorphic sectional curvature. This implies that it is K\"ahler-Einstein, has constant negative scalar curvature, negative orthogonal bisectional curvature and vanishing orthogonal anti-bisectional curvature.
\item As a space of constant holomorphic sectional curvature, it is locally (but not globally) holomorphically isometric to $\mathbb{W}_{\tt Normal}$.
\item It is biholomorphic to a tube domain whose base is  $\{ (\theta^1, \theta^2) ~|~ \exp(\theta^1)+ \exp(\theta^2) < 1 \}$. 
\end{enumerate}

\begin{question}
Is there a global isometric embedding from $\mathbb{M}_{\tt NegTri}$ to a Siegel domain?
\end{question}

In other words, is it possible to consider $\mathbb{M}_{\tt NegTri}$ as a subset of the Siegel domain? We suspect that the answer to this question is negative, but do not have a proof.\footnote{The moduli space of multinomial distributions induces an incomplete space of constant \textit{positive} holomorphic sectional curvature (see Subsection \ref{positively curved statistical mirrors}). This space has a natural immersion into $\mathbb{CP}^n$, but no embedding. It is likely that the same phenomena occurs with negative multinomials as well.}

\subsubsection{Properties of $\mathbb{W}_{\tt NegTri}$}

The dual manifold $\mathbb{W}_{\tt NegTri}$ is induced by the sufficient statistics \eqref{negtrisufficientstatistics} and is defined on the tube domain $T \Omega^\ast_{\tt NegTri}$. Its K\"ahler potential is given by the horizontal lift of the Legendre dual of the log-partition function, which is
\begin{equation} \label{NegtrilogpartitionLegendredual}\Phi_{\tt NegTri}^*(\eta) =   \eta_1 \log (\eta_1) + \eta_2 \log(\eta_2) - (1 + \eta_1 + \eta_2) \log( 1 + \eta_1 + \eta_2).
\end{equation}
We note some important properties of $\mathbb{W}_{\tt NegTri}$.

\begin{enumerate}
\item It is incomplete.
\item The Ricci curvature is negative. Showing this requires a bit of computation. To show this, we first consider the matrix
\[\Ric_{ij} = -\frac{\partial^2 \log[\det [D^2 \Phi_{\tt NegTri}^*(\eta)]}{\partial \eta_i \partial \eta_j}. \]
We then lift one of the indices and consider the matrix 
\[ \Ric_i^j = \Ric_{ik} \, g^{jk}. \]
The determinant of this matrix is
\[ \det[\Ric_i^j] = \frac{1+2 \left(\eta_1+  \eta_1\cdot \eta_1+ \eta_2+ \eta_2 \cdot \eta_2+ \eta_1 \cdot \eta_2\right)}{\eta_1 \eta_2(1+\eta_1+\eta_2)},\]
which is positive. On the other hand, the trace of $\Ric_i^j$ (i.e., half the scalar curvature), is negative.
\[ \Tr[ \Ric_i^j] = -3-\frac{1}{\eta_1}-\frac{1}{\eta_2}+\frac{1}{1+\eta_1+\eta_2}. \]
From this it follows that the eigenvalues of the Ricci curvature are negative, and that the scalar curvature is negative (but not constant).
\item Neither the orthogonal bisectional curvature nor the holomorphic sectional curvature have a sign.
\item It is biholomorphic to a tube domain whose base is the first quadrant, which is biholomorphic to the bidisk in $\mathbb{C}^2$. Although the domain admits a transitive group of automorphisms, most of the automorphisms are not isometries and the metric is not homogeneous . 
\item The spaces $\mathbb{M}_{\tt NegTri}$ and  $\mathbb{W}_{\tt NegTri}$ are not biholomorphic. To see this, note that \[ \Phi(\theta)^h = \left( e^{\theta^1}+e^{\theta^2}-1 \right)^h \] is a pluri-subharmonic exhaustion function on $T \Omega_{\tt NegTri}$ which is strictly pseudo-convex on the boundary $\partial T \Omega$. However, the bidisc is not biholomorphic to any strictly pseudo-convex domain.
\end{enumerate}

\subsubsection{Statistical mirror symmetry and biholomorphic isometries}

This example of statistical mirror symmetry, in conjunction with the one in Section \ref{Normaldistributions}, highlights a property of T-duality that is initially surprising. Upon first encountering the concept, one might imagine that if two semi-flat K\"ahler manifolds are locally biholomorphically isometric, that their dual spaces are also locally isometric. This is the case for compact semi-flat Calabi-Yau manifolds with their Ricci flat metrics, since such spaces are simply complex tori. Such an isometry need not extend globally, and the failure for a local isometry to extend can be seen at the length scale of the co-compact lattice.

However, if we drop the assumption that the metric is Ricci-flat (as done in \cite{fei2019anomaly}), then the situation becomes more complicated. In particular, it is possible to construct semi-flat metrics which have neighborhoods that are locally biholomorphically isometric, but whose dual spaces are not locally isometric. More precisely, if we quotient the fibers of $\mathbb{M}_{\tt NegTri}$ and  $\mathbb{W}_{\tt Normal}$ by a pair of co-compact lattices, the resulting spaces will be locally biholomorphically isometric, since they are spaces with constant holomorphic sectional curvature.\footnote{There are other examples of Hessian metrics whose $r$-map has constant holomorphic sectional curvature. For instance, if we consider the domain $
    \Omega = \{ (x^1, x^2) \in \mathbb{R}^2 ~|~ x^1-\exp(x^2) >0 \}$, the convex potential $ \Phi(x) = -\log(x^1-\exp(x^2)) $ gives another.}$\,$\footnote{As with the case of flat tori, the failure of a global isometry will be apparent at the length-scale of the co-compact lattices, since the dually-flat structures are very different. } On the other hand, $\mathbb{W}_{\tt NegTri}$ and $\mathbb{M}_{\tt Normal}$ are distinct as K\"ahler manifolds and remain so when quotiented by dual co-compact lattices. For instance, $\mathbb{M}_{\tt Normal}$ has constant scalar curvature whereas $\mathbb{W}_{\tt NegTri}$ does not. These spaces are non-compact (and cannot be quotiented to obtain compact examples), but it is possible to locally patch these potentials in a compact Hessian manifold (at the expense of deforming the potential elsewhere) to obtain semi-flat tori $\mathbb{M}_1$ and $\mathbb{M}_1$ which have biholomorphically isometric open subsets but whose T-duals $\mathbb{W}_1$ and $\mathbb{W}_1$ have no isometric neighborhoods.

\subsection{Inverse Gaussian distributions}

We now consider another exponential family whose underlying statistical manifold is hyperbolic, which is the family of inverse Gaussian distributions. This is the family of distributions whose probability density is given by \[ f_{\tt InvGau}(\zeta|~\lambda, \mu) = \sqrt\frac{\lambda}{2 \pi \zeta^3} \exp\left[-\frac{\lambda (\zeta -\mu)^2}{2 \mu^2 \zeta}\right]. \]
where $\lambda$ and $\mu$ are the shape parameter and mean, respectively. 
 As the name suggests, there is a relationship between this distribution and the Gaussian distribution. However, this connection requires some explanation since it involves stochastic processes.

We consider the stochastic process $X(t)$ given by 
\begin{eqnarray*}
X(0) = 0, \qquad \qquad
X(t)  = \nu t + W(t),
\end{eqnarray*}
where $W(t)$ is standard Brownian motion and $\nu$ is a fixed positive constant which denotes the speed of the drift. We then consider the arrival times $T_\alpha$ at a fixed level $\alpha >0$, which are given by 
\[ T_\alpha= \inf \{ t>0 ~|~ X(t) = \alpha \} \] and are distributed as \[ T_\alpha \sim {\tt InvGau} \left( \frac{\alpha}{\nu},\left(\frac{\alpha}{\nu} \right)^2 \right). \]
For any pair of times, $W(t) - W(s)$ will be normally distributed, which gives the connection between inverse Gaussian distributions and Gaussian distributions.

As before, we consider the moduli of inverse Gaussian distributions as a statistical manifold, and calculate its Fisher metric. In the $\mu, \lambda$ coordinates, this is given by 
$$
g =  \left[
\begin{array}{cc}
\frac{1}{2 \lambda^2}  & 0  \\
 0 & \frac{\lambda}{2 \mu^3}  \\
\end{array}
\right] 
$$

When we compute the curvature, we see that $\mathcal{M}_{\tt InvGau}$ has constant negative curvature, so is hyperbolic. Furthermore, the metric is unbounded but incomplete. To understand the global geometry of this space, it is helpful to change of coordinates to $$
x^1= \mu^{-1/2} \,\,\, \mbox{and} \,\,\,  x^2= \lambda^{-2}. $$ In the $(x^1,x^2)$ coordinates, $\mathcal{M}_{\tt InvGau}$ is the standard half-plane model restricted to the first quadrant.

\subsubsection{Symmetries of the moduli of inverse Gaussian distributions}

By studying the geometry of the moduli of inverse Gaussians, it is possible to prove new statistical results. For instance, the previous calculations allow us to understand the symmetries of the inverse Gaussian family.

\begin{theorem}
The isometry group of the inverse Gaussian family is the infinite dihedral group $\mathbb{R} \ltimes \mathbb{Z}_2$ generated by the mappings 
$$(\mu^{-1/2}, \lambda^{-2}) \mapsto (\alpha \mu^{-1/2}, \alpha \lambda^{-2})$$ for $\alpha > 0$ and inversions in the $(x^1,x^2)$-coordinates around a circle centered at the origin.
\end{theorem}

\begin{proof} To see this, it suffices to compute the symmetries of the hyperbolic half-plane which fix the line $x^1 = 0$. This is an infinite dihedral group, generated by scaling and inversions around a circle centered at the origin.
\end{proof}
The fact that scaling generates a symmetry has a natural interpretation as it corresponds to scaling the drift $\nu$ and the level $\alpha$ for the stochastic process $X(t)$. This symmetry is not Fisher-Chentsov, but its physical meaning is clear. However, we are not aware of a physical interpretation for the isometry given by inversions.


\begin{question}
Is there a natural geometric interpretation for the symmetry of the moduli of inverse-Gaussians induced by inversions?
\end{question}

\subsubsection{Properties of $\mathbb{M}_{\tt InvGau}$}

Inverse Gaussian distributions are another example of an exponential family. As such, we can use their moduli space to construct a K\"ahler statistical mirror pair. To do so, we consider the natural parameters 
\begin{eqnarray*}
\theta^1 &=& \frac{\lambda}{2 \mu^2} , \\
\theta^2 &=& - \frac{\lambda}{2} 
\end{eqnarray*}
and use the connection associated with these coordinates. In this case, the log-partition function is given by 
$$ \Phi( \theta) = -\sqrt{  \theta^1  \theta^2} - \frac{1}{2} \log ( -  \theta^2) , $$ 
which is supported on $\Omega = \{  \theta^1,  \theta^2 < 0 \}$.
 The K\"ahler manifold $\mathbb{M}_{\tt InvGau}$ has the following properties.

\begin{enumerate}
\item It has negative Ricci curvature (and hence negative scalar curvature).
\item The holomorphic sectional curvature, antibisectional curvature and orthogonal bisectional curvatures do not have a sign.
\item $\mathbb{M}_{\tt InvGau}$ a tube domain whose base is the third quadrant. As such, it is biholomorphic to the bidisk in $\mathbb{C}^2$.
\end{enumerate}

\subsubsection{Properties of $\mathbb{W}_{\tt InvGau}$}

To construct the dual K\"ahler manifold, we consider the dual connection whose associated affine coordinates are  
\begin{eqnarray*} \eta_1 &=& \frac{ \partial \Phi}{\partial \theta^1}  = - \frac{1}{2} \sqrt{ \frac{\theta^2}{\theta^1}} \\ \eta_2 &=& \frac{ \partial \Phi}{\partial \theta^2} = - \frac{1}{2} \sqrt{ \frac{\theta^1}{\theta^1}} - \frac{1}{\theta^2}. 
\end{eqnarray*}

The Hessian potential in these coordinates is is 
$$ \Phi^\ast(\eta) = \frac{1}{2} \left[ -1 + \log \left( \frac{2 \eta_1}{-1+4  \eta_1  \eta_2} \right) \right] $$ which is defined on the hyperbola  $\Omega^\ast = \{ (\eta_1, \eta_2) \in \R^2 ~|~ -1+ 4  \eta_1 \cdot  \eta_2>0, \eta_1 > 0 \}$. The  K\"ahler manifold $\mathbb{W}_{\tt InvGau}$ is defined on the tube domain $T \Omega^\ast$ and has the following properties.

\begin{enumerate}
\item It has negative Ricci curvature (and so negative scalar curvature).
\item It has negative orthogonal anti-bisectional curvature.
\item Neither the orthogonal bisectional curvature nor the holomorphic sectional curvature have a sign.
\end{enumerate}

\section{Frobenius manifolds and their statistical mirrors}
\label{Flat}

Thus far, we have considered examples where the underlying statistical manifold has negative curvature. In this section, we turn our attention to statistical manifolds which are Riemannian-flat and consider their sttatistical mirrors. The primary motivation for studying such spaces is that they induce solutions to the Witten–Dijkgraaf–Verlinde–Verlinde equations (WDVV equations), which play an important role in mathematical physics.

\subsection{WDVV equations}

The WDVV equations are the system of equations
\begin{equation} \label{WDVVequations}
\sum_{p,q} \Phi^{pq} \Phi_{jlp} \Phi_{ikq} = \sum_{p,q}  \Phi_{ilp} \Phi_{jkq} \Phi^{pq}, 
 \end{equation}
 which are indexed in the variables $i,j,k$ and $\ell$. Here, we use the notation $\Psi_J$ to denote  \[ \frac{ \partial^{| J |} \Psi}{\partial u^J} \]
for a multi-index $J$ and use $\Psi^{ij}$ to denote the $i,j$-th component of the matrix inverse of the Hessian $\Psi_{\alpha \beta}$.
 
A solution of these equations give a flat Riemannian manifold $(\M,g)$ the structure of a \textit{Frobenius manifold}. In other words, a solution to these equations induces a commutative, associative product on the tangent bundle. For a myriad of applications in topological field theory and physics, see \cite{dubrovin1996geometry} \cite{magri2015wdvv} \cite{manin1999frobenius}. 
Furthermore, the WDVV equations are of interest from the perspective of information geometry due to the following correspondence.

\begin{obs*}[\cite{kito1999hessian}]
The potential of a Riemannian-flat Hessian metric solves the WDVV equations.
\end{obs*}

To see this, we can calculate the curvature of a Hessian metric in the associated affine coordinates (see  \cite{totaro2004curvature} for details). Doing so, we find the following:
\begin{equation}
 R_{i j k l}=-\frac{1}{4} \sum_{p, q} \Phi^{p q}\left( \Phi_{j l p} \Phi_{i k q}-\Phi_{i l p} \Phi_{j k q}\right).
\end{equation}

As a word of caution, this quantity is the curvature of the Levi-Civita connection on the underlying Hessian manifold, and not the curvature of the K\"ahler metrics on its tangent bundle (which involve fourth derivatives of the potential as well).
From this, we see that a convex function $\Phi$ solves the WDVV equations if and only if its associated Hessian metric is Riemannian-flat. This correspondence gives a geometric way to understand the WDVV-equations and to derive some of its properties.

For instance, it provides a particularly simple proof that the Legendre transformation of one solution to the WDVV-equations yields a second solution.

\begin{obs*}
Suppose that a convex potential $\Phi: \Omega \to \mathbb{R}$ for $\Omega \subset \mathbb{R}^n$ satisfies Equations \ref{WDVVequations}.
Then the Legendre dual $\Phi^*: \Omega^\ast \to \mathbb{R}$ also satisfies Equations \ref{WDVVequations}.
\end{obs*}

To see this, note that the Hessian manifolds $(\Omega, \Phi)$ and $(\Omega^\ast, \Phi^\ast)$ are isometric as Riemannian manifolds (in fact, $\Omega$ and $\Omega^\ast$ can be interpreted as being two sets of coordinates on the same underlying manifold). Furthermore, we can use this perspective to simplify the WDVV equations in dimensions two and three.

\begin{proposition}
\begin{enumerate}
\item 
A convex potential $\Phi: \Omega \to \mathbb{R}$ for $\Omega \subset \mathbb{R}^2$ satisfies Equations \ref{WDVVequations} if and only if it satisfies the (single) equation
  \begin{equation} \label{Hessianscalar}
     \sum_{i,j,k, l,p,q} \Phi^{jl} \Phi^{ik} \Phi^{pq} \Phi_{jlp} \Phi_{ikq} = \sum_{i,j,k, l,p,q}  \Phi_{ilp} \Phi_{jkq} \Phi^{pq} \Phi^{jl} \Phi^{ik}.
  \end{equation} 
\item
  A convex potential $\Phi: \Omega \to \mathbb{R}$ for $\Omega \subset \mathbb{R}^3$ satisfies Equations \ref{WDVVequations} if and only if it satisfies the six equations (indexed symmetrically in $i$ and $k$)
     \begin{equation} \label{HessianRicci}
    \sum_{j, l,p,q}  \Phi^{jl} \Phi^{pq} \Phi_{jlp} \Phi_{ikq} =  \sum_{j, l,p,q}  \Phi_{ilp} \Phi_{jkq} \Phi^{pq} \Phi^{jl}.   \end{equation} 
\end{enumerate}
\end{proposition}

\begin{proof}
Equation \ref{Hessianscalar} is satisfied if and only if the scalar curvature of a Hessian metric vanishes. For surfaces, the Riemannian curvature is completely determined by the scalar curvature, so the metric is flat whenever the scalar curvature vanishes.

Similarly, Equations \ref{HessianRicci} are satisfied whenever the Ricci curvature of the associated Hessian metric vanishes. For three-folds, the Ricci curvature completely determines the Riemannian curvature, so the metric is flat whenever its Ricci curvature vanishes.
\end{proof}

\subsection{Examples}

As we will see in the next section, Frobenius structures exist in abundance. Here, we have included a partial list of such potentials. For each of these examples, the Legendre dual will induce a separate Frobenius structure as well.

\begin{enumerate}
    \item The simplest example of a Frobenius manifold is $\mathbb{R}^n$ with the quadratic potential
    \[ \Phi = \sum_{i=1}^n \frac{x^i \cdot x^i}{2}. \]
    This is the unique structure which is self-dual.
    \item Another example of a Frobenius structure is $\mathbb{R}^2$ with the Hessian potential
\[ \Phi(x^1,x^2) = \log( \cosh(x^1) + \cosh(x^2)). \]
\item Closely related to the previous example, we can replace the hyperbolic cosines for normal cosines and consider the potential
\[ \Phi(x^1,x^2) = - \log( \cos(x^1) + \cos(x^2)), \]
which is defined on a diamond in $\mathbb{R}^2$. The curvature of the associated K\"ahler manifold is opposite that of the previous example.
\item One well-known example is the domain
\[        \Omega = \{ x \in \mathbb{R}^2 ~|~ x^2 < |x^1|  \} \] with the potential  \[ \Phi(x^1,x^2) = - \log( x^1 \cdot x^1 - x^2 \cdot x^2 ). \]  
This example is of interest because it is possible to quotient this cone by co-compact automorphisms to construct a compact Hessian manifold. This example can be generalized to arbitrary dimensions by reparametrizing the coordinates as $\theta^1=x^1-x^2$ and $\theta^2=x^1+x^2$ and considering the more general potential
\[\Phi= - \sum_{i=1}^n \log(\theta^i) , \] which is defined in the positive cone
\[        \Omega = \{ \theta = (\theta^1, \cdots, \theta^n) \in \mathbb{R}^n ~|~ \theta^i >0  \}. \]

The $x$-parametrization indicates that there may be a connection between this metric and Siegel domains. More precisely, if we consider the potentials  \[ \Phi_\epsilon =  \log \left( \left( \epsilon x^1 +\epsilon^{-1} \right)^2 - \frac{1}{\epsilon^2} - x^2 \cdot x^2  \right),  \] it might be possible to obtain the complex ball as a suitably renormalized limit.

\begin{question}
Is there a precise sense in which the ball can be understood as the limit of cones? If so, are there any applications for this fact?
\end{question}

\end{enumerate}

\subsection{Positive curved statistical mirrors}
\label{positively curved statistical mirrors}
We have now studied statistical mirror symmetry in both the flat case and the hyperbolic case, so it is natural to consider positively curved Hessian manifolds as well. However, when one tries to do so, one encounters the immediate problem that any complete space of strongly positive sectional curvature is not affine.\footnote{To see this, note that Myer's theorem implies that the fundamental group of a manifold with strongly positive curvature is finite, whereas the fundamental group of a compact affine manifold must be infinite \cite{ay2002dually}.} Therefore, it is only possible to construct Hessian metrics on an open subset of spaces  of strongly positive sectional curvature.
In \cite{khan2020kahler}, the authors studied one such example, which was given by the Hessian potential 
\[\Phi(x) = \log \left( 1+ \sum_{j=1}^{n-1} e^{x^j}\right) , \]
which is defined on all of $\mathbb{R}^{n-1}$. This is the dually flat geometry of the probability simplex
$$
\Delta^n = \left\{ (p_1, \cdots, p_n) \in \R^n ~ |~ p_i \geq 0,\, \sum_{i=1}^{n} p_i = 1 \right\}. 
$$
This example is motivated by statistical considerations, since it induced by the exponential family of multinomial distributions. It also has applications in optimal transport and mathematical finance (see \cite{pal2016geometry, pal2018exponentially} for details).  
Interestingly, the K\"ahler metric associated to this potential has constant positive holomorphic sectional curvature and is defined on $\mathbb{C}^n$. Since it is incomplete, this does not contradict the fact that complete K\"ahler manifolds with positive constant holomorphic sectional curvature are biholomorphically isometric to $\mathbb{CP}^n$.

If one drops the assumption of strongly positive curvature, it may be possible to find examples of complete metrics which have positive curvature in some weaker sense. In \cite{khan2020geometry}, the authors and Zheng studied this question and constructed $O(n)$-symmetric metrics with non-negative orthogonal anti-bisectional curvature. However, we showed that there are no non-trivial $O(n)$-symmetric K\"ahler Sasaki metrics with positive bisectional curvature (or orthogonal bisectional curvature for $n \geq 3$).

\section{The existence theory for Hessian metrics}
\label{Existence of Hessian metrics}

Thus far, we have not discussed the existence theory for Hessian metrics in any detail. It is a natural question of whether, given a Riemannian manifold $(\M,g)$ with $\dim(\M) = n$, there exists a flat connection $D$ (either globally or locally) and a potential $\Phi$ so that \begin{equation}
    g = D^2 \Phi.
\end{equation}

From a global perspective, we can immediately see that in order for such a structure to exist, the manifold $\M$ must be affine. That is, it must admit a flat connection. This is a fairly restrictive condition from a topological perspective. For instance, the only compact orientiable surface which admits an affine structure is the torus.\footnote{In fact, the torus admits several affine structures which are inequivalent (see \cite{yagi1981hessian} for details).} However, merely admitting an affine structure is not sufficient for a manifold to admit a Hessian metric. For instance, the Hopf manifolds $S^n \times S^1$ (for $n >1$) do not admit such a structure \cite{yagi1981hessian}. 

It is also possible to study this question \emph{locally}. For this, one considers a Riemannian manifold $(\M,g)$ and a point $x \in \M$. The local question studies whether it is possible to find a connection $D$ which is flat in a neighborhood of $x$ and which satisfies $ D^2 \Phi =g$? 

When the dimension is greater than two, generic Riemannian metrics do not admit a Hessian metric, even locally. Intuitively, this can be seen by counting equations; a Hessian potential and flat connection can be specified by $n+1$ local functions ($n$ functions for the associated affine coordinates and a single function for the convex potential) whereas a generic Riemannian metric can be understood as a map from each point (in a local chart) to the space of $n \times n$ symmetric positive definite matrices, which has dimension $\frac{(n+1)n}{2}$. In fact, for $n \geq 4$, there are point-wise curvature obstructions to the existence of a Hessian metric (see \cite{amari2014curvature} for details). However, in dimension two, Amari and Armstrong proved that given a real analytic metric, it is always possible to find a Hessian structure locally.\footnote{Another proof of this fact was given by Bryant in a MathOverflow answer. \cite{Bryant2013Hessiantype}}

The complement to the existence question is uniqueness. In other words, given a Hessian manifold $(\M,g,D)$, is it possible to find another connection $D^\prime$ so that  $(\M,g,D^\prime)$ is a Hessian manifold? By considering the connection $D^\prime = D^\ast$, we can immediately see that Hessian structures are not unique. In fact, our examples show that there are potentially many different Hessian structures which induce the same Riemannian geometry.

The moduli space of Hessian metrics (for some particular Riemannian metrics) was studied by Kito \cite{kito1999hessian}. In particular, he considered the Euclidean plane and showed that the space of associated Hessian metrics has the freedom of three local functions of $\mathbb{R}$. He also showed that space of Hessian metrics for flat Euclidean $\mathbb{R}^n$ has at least the freedom of $n$ functions. In practice, this means that flat Hessian metrics (and thus solutions to the WDVV equations) exists in abundance. Furthermore, he considered hyperbolic $n$-space and showed that the space of Hessian metrics has at least the freedom of $n-1$ functions on $\mathbb{R}$. This observation naturally raises the following question.

\begin{question}
Are there other dually flat structures (other than the one from normal distributions) defined on the entire hyperbolic plane which are \emph{highly symmetric}? 
\end{question}

There are infinitely many dually flat structure on the hyperbolic plane, though there are no other complete examples whose tube domains have constant holomorphic sectional curvature. Still, it would be of interest to find other examples with large automorphism group or which give rise to a K\"ahler-Einstein metric.

\section{Acknowledgements} 
This projected is supported by DARPA/ARO Grant W911NF-16-1-0383 (``Information Geometry: Geometrization of Science of Information'') and AFOSR Grant FA9550-19-1-0213 (``Brain-Inspired Networks for Multifunctional Intelligent Systems and Aerial Vehicles''), which also supported the first author when he was at the University of Michigan. The first author is currently partially supported by a Simons Collaboration Grant 849022 (``K\"ahler-Ricci flow and optimal transport"). We thank Fangyang Zheng for some helpful conversations on the curvature of K\"ahler manifolds. We also thank Teng Fei for providing comments related to T-duality.

\bibliography{bibfile}
\bibliographystyle{alpha}

\end{document}